\documentclass[11pt,final]{amsart}
\usepackage{math50}
\usepackage{local}

 \theoremstyle{definition}
 
 \theoremstyle{remark}

 \numberwithin{equation}{section}

\usepackage{graphicx}  
\begin{document}

%
%
%
%
%
%
%
%
%

\title[Perturbation of Sectorial Projections]
 {Perturbation of Sectorial Projections of\\ Elliptic
Pseudo-differential Operators}


\author[B. Boo{\ss}--Bavnbek]{\small Bernhelm Boo{\ss}--Bavnbek}

\address{Department of Sciences, IMFUFA\\ Roskilde
University, DK-4000 Ros\-kilde, Denmark}
\email{booss@ruc.dk}
\urladdr{http://milne.ruc.dk/$\sim$booss}

\thanks{The third author was partially supported by the
        Hausdorff Center for Mathematics. 
The fourth author was partially
supported by 973 Program of MOST No. 2006CB805903, Key Project of
Chinese Ministry of Education.(No 106047), IRT0418,  PCSIRT
NO.10621101, LPMC of MOE of China, and Nankai University. }

\author[G. Chen]{Guoyuan Chen}
\address{Chern Institute of Mathematics, Key Lab of Pure Mathematics
and Combinatorics of Ministry of Education,\\ Nankai University,
Tianjin 300071, \\the People's Republic of China}
\email{gychen@mail.nankai.edu.cn}

\author[M. Lesch]{Matthias Lesch}
\address{Mathematisches Institut,\\
Universit\"at Bonn, Endenicher Allee 60, 53115 Bonn, Germany}
\email{ml@matthiaslesch.de, lesch@math.uni-bonn.de}
\urladdr{www.matthiaslesch.de, www.math.uni-bonn.de/people/lesch}
\author[C. Zhu]{Chaofeng Zhu}
\address{Chern Institute of
Mathematics, Key Lab of Pure Mathematics and Combinatorics of
Ministry of Education,\\ Nankai University, Tianjin 300071,\\ the
People's Republic of China}
\email{zhucf@nankai.edu.cn}



\begin{abstract}
Over a closed manifold, we consider the sectorial projection of an
elliptic pseudo-differential operator $A$ of positive order with two
rays of minimal growth. 
We show that it depends continuously on $A$ when the space of pseudo-differential
operators is equipped with a certain topology which we explicitly describe.
Our main application deals with a continuous curve of arbitrary first order linear elliptic
differential operators over a compact manifold with boundary. Under
the additional assumption of the weak inner unique continuation
property, we derive the continuity of a related curve of
Calder\'{o}n projections and hence of the Cauchy data spaces of the
original operator curve.

In the Appendix, we describe a topological obstruction against a
verbatim use of R. Seeley's original argument for the complex
powers, which was seemingly overlooked in previous studies of the
sectorial projection.

\end{abstract}

\maketitle

\section{Introduction}\label{s:intro}

This note describes how the continuity of a curve of operators (here
of sectorial projections) can be derived within the symbolic
calculus, supplemented by estimates of some smoothing operators. As
usual, the smoothing operators appear as correction terms between
pseudo-differential operators and the operators generated by their
total symbol.

The power of the symbolic calculus is well established for the
investigation of spectral invariants, e.g., derived from asymptotics
of the heat kernel. There, perturbations by smoothing operators have
no effect. However, the symbolic calculus may appear as having no
value for establishing the precise continuity of an operator curve
since, {\em a priori}, the variation of the operator norm of
emergent smoothing operators may be hard to control. This note
refutes that view.

\subsection{Various definitions of sectorial projections for elliptic
pseudo-differential operators of positive order}\label{ss:vrious}

\subsubsection{The bounded and the closed self-adjoint cases}\label{sss:bounded}
Let $\mathcal{B}(H)$ denote the space of bounded operators in a
complex separable Hilbert space $H$ and let $A\in \mathcal{B}(H)$.
Assume that there exists a curve $\gG_+\<\C\setminus \spec A$ that
divides $\C$ into two sectors $\gL_\pm$ as in Figure
\ref{f:gamma-plus}a below. Then we can encircle all spectral points in the
positive sector $\gL_+$ by a closed curve $\gG_0$, as in Figure
\ref{f:gamma-plus}b, and so get a well-defined projection, the {\em
sectorial projection}, by setting
\begin{equation}\label{e:bounded}
P_{\gG_+}(A)\, : \, = \, \frac {-1}{2\pi i}\int_{\gG_0} (A-\gl)\ii
d\gl\,.
\end{equation}
From the integral it is clear that
\begin{equation}\label{e:bounded-estimate}
\|P_{\gG_+}(A+B)-P_{\gG_+}(A)\|< C_A\|B\| \text{ for any small
bounded perturbation $B$},
\end{equation}
i.e., the map $P_{\gG_+}:A\mapsto P_{\gG_+}(A)$ is continuous in the
operator norm of $\mathcal{B}(H)$.

The general functional analytical arguments break down for (graph
norm) continuous curves of densely defined closed operators in $H$.
Actually, in our \cite[Example 3.13]{BooLesZhu:CPD} examples of
operators with unbounded sectorial projection were discussed. From
the example it becomes clear that additional assumptions will be
required.

Most easy is to require that $A$ is self-adjoint: Consider a (graph
norm) continuous curve in the space $\mathcal{C}^\sa(H)$ of densely
defined closed self-adjoint operators in $H$. The preceding
perturbation argument generalizes immediately to this case under the
additional condition, that the {\em Riesz transformation}
$F:\mathcal{C}^{\sa}\to \mathcal{B}^{\sa}\/, A\mapsto
F(A):=(I+A^2)^{-1/2}A$ is continuous. It may be worth mentioning
that a counter example (a graph-norm convergent sequence of
unbounded self-adjoint Fredholm operators with divergent Riesz
transforms) was given to us by B. Fuglede several years ago and
elaborated in our \cite[Example 2.14]{BooLesPhi:UFO}. The condition
is satisfied, however, for formally self-adjoint elliptic
differential operators on closed manifolds: they have a discrete
spectrum of finite multiplicity contained in $\R$ and a complete set
of eigenvectors. So, the imaginary axis (or a parallel $\{c + ri\mid
r\in\R\}$ with $c\not\in \spec A$) becomes a suitable separating
curve $\gG_+$ and we obtain
$P_{\gG_+}(A)=1_{[c,\infty)}(A)=1_{[F(c),\infty)}(F(A))$ as a
pseudo-differential projection (the {\em Atiyah-Patodi-Singer} (APS)
projection) by applying the integral representation of
\eqref{e:bounded} to the bounded Riesz transform
 of $A$. Note that $F(A)$ has its spectrum
contained in the interval $(-1,1)$, but has the same eigenspaces and
sectorial projection as $A$. We refer to our \cite[Propositions
7.14-7.15]{BooLesZhu:CPD} (see also \cite[Thm. 4.8]{BooFur:MIF} for
a wider purely functional analytic setting) for a proof of the
continuity of the Riesz transformation $A\mapsto F(A)$ on the space
of self-adjoint elliptic differential operators. That yields the
well-known continuous variation of the APS projection under
continuous variation of the underlying operator as long no
eigenvalue crosses the line $\gG_+$\/, i.e., the continuity of the
map $A\mapsto P_{\gG_+}(A)$, when we take the operator norm $L^2\to
L^2$ for $P_{\gG_+}(A)$ and the operator norm $H^m\to L^2$ for $A$,
where $m$ denotes the order of $A$.

\begin{figure}
\hfill
\begin{minipage}{0.55\textwidth}
\begin{center}
\includegraphics[scale=1.]{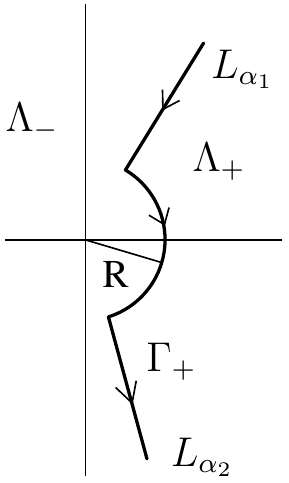}
\end{center}
\end{minipage}
\hfill
\begin{minipage}{0.4\textwidth}
\includegraphics[scale=1.]{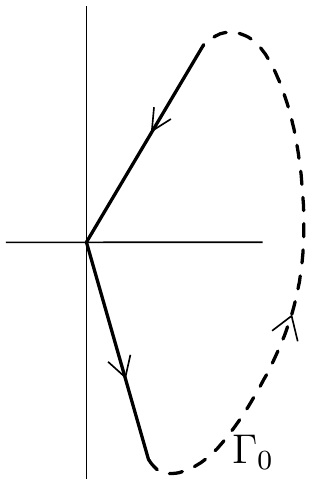}
\end{minipage} \hfill \caption{Left:
{Two} rays of minimal growth and an arc, making the spectral cut
curve $\Gamma_+$\/. Right:
Specifying a bounded set of eigenvalues by a separating
curve $\Gamma_+$ made of two rays and capturing it by a closed
contour $\Gamma_0$.}
 \label{f:gamma-plus}
\end{figure}

\subsubsection{Spectral integrals for elliptic pseudo-differential operators of positive order}
\label{sss:spectral-integrals}
It seems that no general functional analysis methods are available to obtain continuous curves of
sectorial projections for arbitrary continuous curves of operators with compact resolvent and
two rays of minimal growth, if
the operators are neither bounded nor self-adjoint.

As explained in our \cite[Section 3.2]{BooLesZhu:CPD}, a semigroup
$\{Q_+(x,A)\}_{x>0}$ of sectorial operators can be defined by
inserting a weight $e^{-\gl x}$ into the integral \eqref{e:bounded}.
Then {\em sectorial projections} can be defined asymptotically in an
abstract Hilbert space framework. More precisely, for a closed, not
necessarily self-adjoint operator $A$ in separable Hilbert space
with compact resolvent and minimal growth of the resolvent in a cone
we may take the closure of the densely defined $\lim _{x\to
0+}Q_+(x,A)$. However, such projections are unbounded operators, in
general, and do not necessarily vary continuously under perturbation
of the underlying operator, see, once again, \cite[Example
3.13]{BooLesZhu:CPD}. Consequently, one
has to exploit the {\em symbolic calculus} for the investigation of
sectorial projections of not necessarily self-adjoint elliptic
pseudo-differential operators of positive order with two rays of
minimal growth of the resolvent.

Actually, in a slightly different context (namely dealing with
well-posed boundary problems), it was already noticed in
\textsc{Burak} \cite{Bur:SPE} that sectorial projections are bounded
operators. For an elliptic pseudo-differential operator of positive
order over a smooth closed (compact and without boundary) manifold,
that approach was worked out in \textsc{Wodzicki} \cite{Wo:SANR,
Wo:C} and the more recent \textsc{Ponge} \cite{Pon:SAZ} and {\sc
Gaarde} and {\sc Grubb} \cite{GaaGru:LSP}. In some of these papers,
the positive sectorial projection plays a prominent role in more
refined questions related to spectral asymmetry.

Before preceding, we fix the notation:
\begin{convent}\label{c:convent}
{\bf (a)} Let $M$ be an $n$-dimensional closed Riemannian manifold
and $\pi:E\rightarrow M$ a Hermitian vector bundle. Let $A:\Ci(M;E)
\to \Ci(M;E)$ be an elliptic pseudo-differential operator of order
$m>0$.

{\bf (b)} Let $\spec(A)$ denote the spectrum of $A$ regarded as an
operator in $L^2(M;E)$ with the Sobolev space $H^m(M;E)$ as its
domain. We recall that $\spec(A)$ is either the whole complex plane
or a discrete subset of $\C$. The reason is simply that the
resolvent, if it exists, is compact (see \textsc{Shubin}
\cite[Theorem 8.4]{Shu:POS}, similarly already in \textsc{Agmon}
\cite[Section 2]{Agm:EEG} for well-posed elliptic boundary value
problems). Clearly, $\ind A \ne 0$ implies $\spec A=\C$.

{\bf (c)} Let $L_{\alpha_1}=\{\lambda\in \mathbb{C}\mid {\rm
arg}\,\lambda=\alpha_1\}$ and $L_{\alpha_2}=\{{\lambda\in
\mathbb{C}\mid \rm arg}\,\lambda=\alpha_2\equiv \alpha_1-\theta \mod
2\pi \}$ ($0<\theta< 2\pi$) be two rays. We assume that $A$ has only
a finite number of eigenvalues on the rays $L_{\alpha_j}, j=1,2$.

{\bf (d)} Let $\gL := \{re^{i\ga}\mid r<2\rho \text{ or } |\ga
-\ga_j|< \gve, j=1,2\}$ for $\rho,\gve >0$ and $\gve$ sufficiently
small. We can choose $\rho$ in such a way that there exists an $R\in
[0,\rho]$ such that $A-\gl$ is invertible for $\gl\in\gL$ with
$|\gl|\ge R$, and there is only a finite number of eigenvalues in
the region $\gL_R:=\{\gl\in\gL\mid |\gl|<R\}$. For an elaboration of
the meaning of such spectral cuttings, also called {\em rays of
minimal growth} (of the resolvent $(A-\gl)\ii$), see Subsection
\ref{ss:definition} below. If $A$ is differential, then $A-\gl$ is
{\em elliptic with respect to the parameter} $\gl\in \gL$ for
sufficiently small $\rho,\gve >0$ (for that concept c.f.
\textsc{Seeley} \cite{See:CPE} or \textsc{Shubin} \cite{Shu:POS}).

{\bf (e)} Now we choose the curve
\begin{equation}\label{curve}
\Gamma_+=\left\{re^{i\alpha_1}\mid \infty>r\geq R\right\} \cup
\left\{Re^{i(\ga_1-t)}\mid 0  \leq t\leq \theta\right\} \cup
\left\{re^{i\alpha_2}\mid R\leq r<\infty\right\}
\end{equation}
in the resolvent set of $A$, see Figure \ref{f:gamma-plus}a.

{\bf (f)} We define an operator in the following form:
\begin{equation}\label{prej}
P_{\gG_+}(A)=-\frac{1}{2\pi
i}A\int_{\Gamma_+}\lambda^{-1}(A-\lambda)^{-1}d\lambda.
\end{equation}

\end{convent}

{\em A priori}, the integral \eqref{prej} gives rise to an unbounded operator.
A common error in the literature is a {\em verbatim} use of the arguments of
\cite[Theorem 3]{See:CPE} to prove that $P_{\gG_+}(A)$ is a pseudo-differential
operator of order 0,
see \textsc{Wojciechowski} \cite{Woj:SFG}, uncritically reproduced in \textsc{Nazaikinskii et
al.} \cite{NSSS:SBV}, \textsc{Ponge} \cite[Proof of Proposition 3.1]{Pon:SAZ}, and
{\sc Gaarde} and {\sc Grubb} \cite{GaaGru:LSP}. In the Appendix, we shall explain the
topological obstructions that make the argument defective.

Anyway, in \cite[Proof of Proposition 4.1]{Pon:SAZ}, a beautiful
formula is proved, which, according to \textsc{Ponge} was already
observed by \textsc{Wodzicki} in 1985: Assume the preceding
conventions and, moreover, that  $A$ is a classical
pseudo-differential operator. Then we have
\begin{equation}\label{e:wodzicki}
A^s_{\ga_2}-A^s_{\ga_1}=(1-e^{2i\pi s})P_{\gG_+}(A)A^s_{\ga_2}\quad\text{ for all $s\in\C$}.
\end{equation}
The formula relates the complex powers (which were well established
as pseudo-differential operators in \cite{See:CPE}) and the
sectorial projection. Multiplying by $A^{-s}_{\ga_2}$ from the right
yields that the sectorial projection is a pseudo-differential
projection.


\subsection{Perturbations of sectorial projections}

Unfortunately, the authors of this note were not able to derive a
true perturbation result for the sectorial projections from
\textsc{Wodzicki}'s formula. Clearly, other authors before us have
noticed the delicacy of the variation of complex powers of an
operator with a ray of minimal growth. For studying variations of
trace formulas, \textsc{K. Okikiolu} \cite[Section 4]{Oki:CHT},
e.g., defines a {\em symbol-smooth} family of pseudo-differential
operators by the smoothness of the (total) symbol in the usual $\Ci$
Fr{\'e}chet topology. Then she is able to prove the
symbol-smoothness of the complex powers, and the symbol-smoothness
of the sectorial projections follows by \eqref{e:wodzicki}.
\textsc{Okikiolu}'s approximative approach is related to \textsc
{H{\"o}rmander}'s symbolic construction of an ``almost" \Calderon\
projection in \cite[Theorem 20.1.3]{Hor:ALPIII}. Admittedly, such
approximative constructions can be of great value in some contexts.
Actually, \textsc{Okikiolu}'s result suffices for proving her trace
variation results. More precisely, in \cite[Section 4]{Oki:CHT}, 
Okikiolu defined a topology on the space of pseudo-differential
operators by the smoothness of `total' symbols in any local
coordinates. This topology is sufficient for her purpose to study
the variations of trace formulas: it can not miss any smooth
operator whose trace is not zero. However, for our aim to study the
perturbation of sectorial projections in the operator norms, it
seems that the topology defined by Okikiolu is not sufficient. It
may miss the smoothing operators with support far away from the
diagonal of $M\times M$. Such smoothing operators have zero traces,
but may have large operator norms. In this note, we define a locally
convex topology, which treats all the lower terms in bulk and does
not concern the local charts.

Similarly, we can neglect even compact perturbations when we are
interested in index theory. For addressing uniform structures and
perturbation results in that direction we refer to \textsc{Eichhorn}
\cite{Eic:ITG} who presents a systematic study of compact
perturbations of generalized Dirac operators. However, for deciding
about the continuous or not continuous variation of the true
sectorial projections, we are not allowed to neglect contributions
from smoothing operators. We recall that exactly such operators
appear as error terms between pseudo-differential operators and
their approximative symbolic representation.

To obtain estimates for the precise sectorial projections, we shall
therefore not follow the elegant approach by \textsc{Burak} and
\textsc{Wodzicki} et al. \footnote{Having been informed about our
results, Prof. Grubb has, perhaps rightly, pointed out to us, that
sharper perturbation results might be achievable, as it often
happens {\em a posteriori} in mathematics, by alternative methods:
namely, by exploiting the description of the sectorial projection by
a difference of two logarithms in \cite[Prop. 4.4]{GaaGru:LSP}.
Unfortunately, no details regarding the perturbation problem were
communicated by Prof. Grubb in her subsequent comments \cite{Gru:PC}
to an earlier arXiv version of this note.} We choose an approach
which does not require any technology beyond symbolic calculus and
standard estimates for integral operators. From a technological
point of view, our approach may appear less elegant than using
logarithms or complex powers, but it is elementary, transparent and
self-contained - and it works. The delicacy of \textsc{A. Axelsson,
S. Keith and A. McIntosh} \cite{AKM:QEF} indicates that there is no
easy way through to be expected. They studied the Hodge-Dirac
operator $D_g$ defined on a closed Riemannian manifold with metric
$g$. In general, $D_g$ is non-self-adjoint, and its spectrum is
contained in an open double sector which includes the real line.
They showed --  by harmonic analysis methods -- that the spectral
projections of the Hodge-Dirac operator $D_g$ depend analytically on
$L_{\infty}$ changes in the metric $g$.

Hence, {\em a priori}, the variational properties of the symbol do
not suffice for establishing the continuous variation of
$P_{\gG_+}(A)$ in the topology of the operator norm of a suitable
Sobolev space. A smoothing operator may have a large operator norm
defined on any Sobolev space. Therefore, our approach is inspired by
\textsc{Seeley}'s \cite[Corollary 2]{See:CPE}. We only need slightly
sharper estimates than \textsc{Seeley}'s original work.

\subsection{Main result}

\subsubsection{Topology and formulation of main theorem}

Let ${\rm Ell}^m_{\Gamma_+}(M,E)$ denote the space of all elliptic
{\em principally classical} pseudo-differential operators $A$ of
order $m>0$ on $M$ acting on sections of the bundle $E$ such that
the leading symbol $a_m$ of $A$ has no eigenvalues on the two rays
$L_{\ga_j}, j=1,2$ and $A$ no eigenvalues on $\Gamma_+$. Taking for
granted that the spaces $\CL^m(M,E)$ and $\pdo^{m-1}(M,E)$ are well
known (we also recall them in Section \ref{ss:pdos}), we use the
notation {\it principally classical pseudo-differential operators}
$\pdo_{\SC} ^m(M, E) := \CL^m(M,E) + \pdo^{m-1}(M,E)$ for standard
pseudo-differential operators with a homogeneous principal symbol,
where the principal symbol denotes the class of the operator modulo
operators of lower order (for details see below Section
\ref{ss:Semiclassical}). Hence, while we do not restrict our
estimates to classical pseudo-differential operators, we must
require a homogeneity of the principal symbol.

We equip the space ${\rm Ell}^m_{\Gamma_+}(M,E)$ with the locally
convex topology $\mathcal T$ described in Section
\ref{ss:Semiclassical} below. It is not surprising that our topology
is stronger than the usual operator topology between Sobolev spaces
for pseudo-differential operators of fixed order $m$ on closed
manifolds, see \textsc{Atiyah, Singer} \cite{AtSi:IEOI}. For our
applications, continuous or smooth variation of all the symbols does
not suffice. In addition, we shall require that all the derivatives
of the principal symbol vary continuously. The necessity of rather
restrictive requirements for the variation of the highest order
coefficients was indicated in our \cite[Section 7]{BooLesZhu:CPD}
where we emphasized the elementary character of perturbations of
lower order and the delicacy of perturbations of highest order for
the continuous variation of the \Calderon\ projection.

In Section \ref{ss:operator-type}, we give an elementary proof 
for each $A\in {\rm Ell}^m_{\Gamma_+}(M,E)$, that the operator
$P_{\gG_+}(A)$ is well defined by \eqref{prej} as a bounded operator
on $H^s(M;E), s\in\R$. The following theorem is our main result:

\begin{theorem}\label{mt}
With respect to the topology $\mathcal T$, the map
\begin{equation}\label{c}
P_{\gG_+}: {\rm Ell}^m_{\Gamma_+}(M,E)\rightarrow
\mathcal{B}(H^s(M;E)),\quad A\mapsto P_{\gG_+}(A)
\end{equation}
is continuous. Here $\mathcal{B}(H^s(M;E))$ denotes the set of
bounded linear operators on $H^s(M;E), s\in\R$.
\end{theorem}

\subsubsection{The structure of this note}
In Section \ref{s:def} we introduce principally classical symbols
and principally classical pseudo-differential operators
$\pdo_{\SC}^m$ on closed manifolds, define a locally convex topology
on it, and discuss the natural factorization of sectorial
projections. In Section \ref{s:local} we identify the (semi)-norms
we need on $\pdo_{\SC}^m$ to ensure that $P_{\gG_+}$ is continuous.
In Section \ref{s:mtl} we give a technical lemma which is crucial in
the proof of our main theorem. We prove some more estimates which
possibly are of more general interest, as well. In Section
\ref{s:applications} we apply our estimates to the perturbation
problem for sectorial projections and draw some consequences for
index correction formulas and the variation of Cauchy data spaces on
manifolds with boundary. In the Appendix \ref{s:appendixA}, we give
the details of the proof of the technical lemma. In the Appendix
\ref{s:appendixB} we explain the topological obstructions that are
seemingly overlooked in the literature on symbolic calculus.

\section{Definitions and notations}\label{s:def}

To fix the notation, we summarize the basic concepts of symbolic
calculus and introduce a space of 
principally classical pseudo-differential operators on closed
manifolds. Moreover, we fix a locally convex topology on it. For
elliptic principally classical pseudo-differential operators of
positive order and for a fixed contour $\gG_+$ we define the
sectorial projections.

\subsection{Classes of symbols}\label{ss:Symbols}

Let $U\subset \R^n$ be an open subset. We denote by $\sym^m(U\times
\R^n)$, $m\in \R$, the space of (complex valued) symbols (the
generalization for matrix valued symbols is straightforward) of
H\"ormander type $(1,0)$ (\textsc{H\"ormander} \cite{Hor:FIOI},
\textsc{Grigis--Sj{\"o}strand} \cite{GriSjo:MAD}). More precisely,
$\sym^m(U\times \R^n)$ consists of those $a\in \CC^\infty(U\times
\R^n)$ such that for multi--indices $\alpha,\gamma\in \Z_+^n$ and
compact subsets $K\subset U$ we have an estimate
\begin{equation}\label{eq:3.1}
    \bigl|\partial_x^\alpha\partial_\xi^\gamma a(x,\xi)\bigr|
  \le C_{\alpha,\gamma,K} (1+|\xi|)^{m-|\gamma|}, \quad x\in K.
\end{equation}
The best constants in \eqref{eq:3.1} provide a set of semi-norms
which endow $\sym^\infty (U\times
\R^n):=\bigcup_{m\in\R}\sym^m(U\times \R^n)$ with the structure of a
Fr{\'e}chet algebra.

The space $\Csym^m(U\times \R^n)$ of \emph{classical symbols}
consists of all $a\in\sym^m(U\times \R^n)$ that admit sequences
$a_{m-j}\in \cinf{U\times\R^n}, j\in \Z_+$ with
\begin{equation}\label{eq:classical}
   a_{m-j}(x,r\xi)=r^{m-j} a_{m-j}(x,\xi),\quad r\ge 1, |\xi|\ge 1,
\end{equation}
such that
\begin{equation}\label{eq:classical-a}
  a-\sum_{j=0}^{N-1} a_{m-j}\in\sym^{m-N}(U\times \R^n)
\quad  \text{ for all $N\in\Z_+$}.
\end{equation}
The latter property is usually abbreviated
$a\sim\sum\limits_{j=0}^\infty a_{m-j}$.

Homogeneity and smoothness at $0$ contradict each other except for
monomials. Our convention is that symbols should always be smooth
functions, thus the $a_{m-j}$ are smooth everywhere but homogeneous
only in the restricted sense of Eq.\ \eqref{eq:classical}.

Furthermore, we denote by
$\sym^{-\infty}(U\times\R^n):=\bigcap_{a\in\R}\sym^{a}(U\times\R^n)$
the space of \emph{smoothing symbols}.

\subsection{(Classical) pseudo-differential operators}\label{ss:pdos}                

Let $M$ be a smooth manifold of dimension $n$. For convenience and
to have an $L^2$--structure at our disposal, we assume that $M$ is
equipped with a Riemannian metric. We denote by $\pdo^\bullet(M)$
the algebra of {\em pseudo-differential operators} with symbols of
H\"ormander type $(1,0)$ (\cite{Hor:FIOI}, \cite{Shu:POS}), see
Subsection \ref{ss:Symbols}. The subalgebra of {\em classical
pseudo-differential operators} is denoted by $\CL^\bullet(M)$. These
operator algebras are naturally defined on the manifold $M$ by
localizing in coordinate patches in the following way:

Let $U\subset \R^n$ be an open subset. Recall that for a symbol
$a\in\sym^m(U\times\R^n)$, the {\em canonical} pseudo-differential
operator {\em associated} to $a$ is defined by
\begin{equation}\label{eq:psido}%
\begin{split}
 \bigl(\Op(a)\, u\bigr) (x) &:= \int_{\R^n} \, e^{i \langle x,\xi \rangle} \,
                     a(x,\xi) \, \hat{u} (\xi ) \, \dbar \xi \\
                     &= \int_{\R^n}\int_U \, e^{i \langle x-y,\xi \rangle} \,
                       a(x,\xi) \, u(y)  dy \dbar \xi,
\end{split}\qquad \dbar\xi:=(2\pi)^{-n} d\xi.
\end{equation}
For a manifold $M$, elements of $\pdo^\bullet(M)$ (resp.
$\CL^\bullet(M)$) can locally be written as $\Op(\sigma)$ with
$\sigma\in \sym^\bullet(U\times \R^n)$ (resp.
$\CS^\bullet(U\times\R^n)$).

Recall that there is an exact sequence
\begin{equation}\label{e:classical-exact}
0\longrightarrow \CL^{m-1}(M)\longhookrightarrow
\CL^m(M)\xrightarrow{\sigma_m} \Ci(S^*M)\longrightarrow 0\,,
\end{equation}
where $\sigma_m(A)$ denotes the {\em principal} (homogeneous
leading) symbol of $A\in\CL^m(M)$. Here, $S^*M$ denotes the cosphere
bundle, i.e., the unit sphere bundle $\subset T^*M$. As usual, the
principal symbol is locally defined as a map
$\sigma_m:\sym^m(U\times\R^n)\to\Ci(U\times S^{n-1})$ by putting
\begin{equation}\label{e:principal-symbol}
\sigma_m(x,\xi):=\lim_{r\to\infty}r^{-m}a(x,r\xi).
\end{equation}
Note that $\sigma_m(A)$ is a homogeneous function on the symplectic
cone $T^*M\setminus M$. We will tacitly identify the homogeneous
functions on $T^*M\setminus M$ by restriction with $\cinf{S^*M}$.

Recall that the principal symbol map is multiplicative in the sense
that
\begin{equation}\label{eq:SymbolMultiplicative}
 \sigma_{a+b}(A\circ B)=\sigma_a(A)\sigma_b(B)
\end{equation}
for $A\in\CL^a(M), B\in\CL^b(M)$.

\subsection{Principally classical pseudo-differential operators}\label{ss:Semiclassical}
As mentioned in the Introduction,
continuous variation of the operator $A$ by bounded $L^2\to L^2$
perturbation is sufficient to obtain continuous variation of the
Cauchy data space, of the Calder{\'o}n projection and of the
sectorial projection in various cases (see \cite[Theorem
3.8]{BooFur:MIF}, \cite[Proposition 7.13]{BooLesZhu:CPD}). However,
we have a hunch that continuous variation of the operator $A$ in the
operator norm, say from $H^{m}(M)$ to $L^2(M)$ will not always yield
continuous variation of the sectorial projection $P_{\gG_+}(A)$ in
the operator norm from $L^2(M)$ to $L^2(M)$. These are our intuitive
arguments:

We know that general functional analysis does not suffice to obtain
the boundedness of the sectorial projection. The more refined
structure of differential or pseudo-differential operators is
required. Apparently, for variation in the highest order, the
principal symbol must be singled out. All that indicates that
variation in the operator norm hardly will suffice for continuous
variation of the sectorial projection.

We use the following convention which
will be in effect for the rest of this note:
\begin{convent}\label{c:operator-norm}
We denote the norm on the space $\cB(H^s,H^t)$ of bounded operators
from the Sobolev space $H^s(M;E)$ to $H^t(M;E)$ by
$\|\cdot\|_{s,t}$\,.
\end{convent}

Let ${\rm L}^{m-1}(M,E)$ (resp. ${\rm CL}^m(M,E)$) denote the space
of $(m-1)$th order pseudo-differential operators on $E$ (resp. $m$th
order classical pseudo-differential operators). Set ${\rm L}_{{\rm
pc}}^m(M,E):={\rm CL}^m(M,E)+{\rm L}^{m-1}(M,E)$. We call it the
space of {\it principally classical pseudo-differential operators}.
Since ${\rm CL}^m(M,E) \cap {\rm L}^{m-1}(M,E) = {\rm
CL}^{m-1}(M,E)$, the principal symbol map
$$\sigma_m:{\rm L}_{{\rm pc}}^m(M,E)\to C^{\infty}(S^*M;{\rm End}(\pi^*E))$$
is well-defined. Here $\pi:S^*M\to M$ denotes the canonical
projection map. We fix a right inverse
$${\rm Op}:C^{\infty}(S^*M;{\rm End}(\pi^*E))\to {\rm L}_{{\rm pc}}^m(M,E)$$
of $\sigma_m$, obtained by patching together the local {\rm Op}-maps
(\ref{eq:psido}) via a partition of unity. Define a map
\begin{equation}\label{e:splitting1}
\begin{matrix}
C^{\infty}(S^*M;{\rm End}(\pi^*E))& \oplus & {\rm L}^{m-1}(M,E) & \rightarrow & {\rm L}_{{\rm pc}}^m(M,E)\\
a &\oplus & B &\mapsto& {\rm Op}(a)+B.
\end{matrix}
\end{equation}
It is a bijection with the inverse
\begin{equation}\label{e:splitting2}
\begin{matrix}
{\rm L}_{{\rm pc}}^m(M,E)& \rightarrow & C^{\infty}(S^*M;{\rm End}(\pi^*E))& \oplus & {\rm L}^{m-1}(M,E)\\
T &\mapsto& {\rm Op}(\sigma_m(T)) &\oplus& T- {\rm Op}(\sigma_m(T)).
\end{matrix}
\end{equation}
We topologize the right hand side of \eqref{e:splitting2} as follows:
\renewcommand*{\labelenumi}{%
   \arabic{enumi}.}
\begin{enumerate}

\item On $\pdo^{m-1}(M,E)$ we take the countably many semi-norms $\|T\|_{k+m-1,k}\,$ for
$k\in\Z$.

\item The summand  $\Ci(S^*M; \End(\pi^*E))$ is equipped with the $\Ci$-topology.
This is known to be a Fr{\'e}chet-topology, hence is generated by
countably many semi-norms $\bigl(p_j\bigr)_{j\in\Z_+}$\/.

\end{enumerate}

\begin{dfn}\label{d:topology}
The locally convex topology on $\pdo_{\SC}^m(M,E)$ induced by the
countably many semi-norms $\|\cdot\|_{k+m-1,k}\, ,\ k\in\Z$ and
$p_j\/,\ j\in\Z_+$ is denoted by $\cT$.
\end{dfn}

It follows from complex interpolation that for each {\em real} $s$
the (semi)-norm $\|~\cdot~\|_{s+m-1,s}$ is continuous with regard to
$\cT$. Furthermore, it is straightforward to see that $\cT$ is independent of the choice
of $\Op$. Moreover, it is worth noting that $\cT$ is not complete. By construction, the
completion of
 $\pdo_{\SC}^m(M,E)$ is a Fr{\'e}chet space which is of the form
\[
\operatorname{CZ}^{m-1}(M,E) \oplus \Ci(S^*M; \End(\pi^*E)).
\]
Here $\operatorname{CZ}^{\bullet}(M,E)$ is (a variant of) the
well-known Calder{\'o}n-Zygmund graded algebra (cf. \cite[Chapter
16]{Pal:SASIT}).

\begin{rem}\label{r:phi-zero}
We record that a sequence $(T_n)_{n\in\N}\<\pdo_{\SC}^m(M,E)$
converges to $T\in \pdo_{\SC}^m(M,E)$ if and only if
\renewcommand*{\labelenumi}{%
   (\roman{enumi})}
\begin{enumerate}

\item $\sigma_m(T_n)\too \sigma_m(T)$ in the $\Ci$-topology of
 $\Ci(S^*M; \End(\pi^*E))$, and

\item $T_n-\Op(\sigma_m(T_n))\too T-\Op(\sigma_m(T))$ with regard to $\|\cdot\|_{k+m-1,k}$
for all $k\in\Z$\/.
\end{enumerate}
\end{rem}

\subsection{The definition of $P_{\gG_+}(A)$}\label{ss:definition}
We shall give our
definition in some detail. These details will be decisive for
proving the perturbation results.

\subsubsection{Our data}\label{ss:data}
Our data are as in Convention \ref{c:convent}. More specifically, we
shall assume $A\in \pdo_{\SC}^m(M,E)$ and that the principal symbol
$a_m(x,\xi)$ of $A$ has no eigenvalues on the rays $L_{\alpha_j},
j=1,2$ for each point $x\in M$ and covector $\xi\in T^*_xM, \xi\ne
0$. For simplicity, denote the principal symbol $\sigma^m_A(x,\xi)$
by $a_m=a_m(x,\xi)$. Note that every ray of minimal growth has a
cone-shaped neighborhood  $\Lambda$ such that any ray contained in
 $\Lambda$ is also a ray of minimal growth for $A$. Then there exists $R>0$
 such that Conventions \ref{c:convent}.(c)-(d) are satisfied and $A-\lambda$ is
invertible for $\lambda\in  \Lambda,\,|\lambda|>R$. Moreover, we
have
 \begin{equation}\label{-1}
\|(A-\lambda)^{-1}\|_{s,s+p}\leq C|\lambda|^{-1+\frac{p}{m}},\qquad
0\leq p\leq m,\,s\in \mathbb{R},
 \end{equation}
for any such $\lambda$. 
For the proof of \eqref{-1} see \cite[Corollary
1]{See:CPE}. For differential operators see also \cite[Theorem
9.3]{Shu:POS}. Equation \eqref{-1} explains the common usage of
``ray of minimal growth of the resolvent" for such spectral cutting
rays.

\subsubsection{Definition of the sectorial
projection and our goal}\label{ss:sectorial-definition}
 Equation \eqref{-1} explains why we
cannot expect convergence of the integral
$\int_{\Gamma_+}(A-\lambda)^{-1}\,d\lambda$, which is familiar in
the bounded case presented above in \eqref{e:bounded}. The common
way to get something finite is to guarantee convergence of the
integral by inserting a factor $\gl\ii$ and to compensate by
multiplying the integral by $A$.

\begin{dfn}\label{d:sectorial-projection}
For the preceding data, we define
\begin{equation}\label{def}
P_{\gG_+}(A):=\frac {-1}{2\pi i} A\,\Phi(A), \qquad
\Phi(A):=\int_{\Gamma_+}\lambda^{-1}\,(A-\lambda)^{-1}\,d\lambda.
\end{equation}
\end{dfn}

\begin{rem}\label{r:goal} (a) In view of the estimate \eqref{-1},
the composition of $A$ with the integral $\Phi(A)$ \emph{a priori}
gives rise to an unbounded operator on $L^2(M;E)$ with domain
$\cup_{s>0}H^s(M;E)$. The nice fact, however, is that $P_{\gG_+}(A)$
truly is a bounded operator (see Section \ref{ss:operator-type}).

\noi (b) Our goal is to prove that with respect to the topology $\cT$
\begin{equation}\label{e:main-mapping}
P_{\gG_+}: \Ell^m_{\gG_+}(M,E)\too \cB(H^s(M;E)) \quad\text{ is
continuous for all $s\in\R$}.
\end{equation}
Here we keep the rays $L_{\ga_j}, j=1,2$ and the contour $\gG_+$
fixed and set
\begin{multline}\label{e:ell-gamma-plus}
\Ell_{\gG_+}^m(M,E):=\{A\in \pdo^m_{\SC}(M,E)\mid A \text{ elliptic,
}
\spec A\cap \gG_+=\emptyset\\
\text{ and $L_{\ga_j}, j=1,2$ rays of minimal growth}\}.
\end{multline}

\noi (c) As a side result, we shall show under what conditions
$P_{\gG_+}(A)$ becomes a pseudo-differential operator. We consider
that of minor importance. The proof of (b) will anyway show that
$P_{\gG_+}(A)$ is of the form $P_{\gG_+,0}(A)+K$ with
$P_{\gG_+,0}(A)\in \pdo_{\SC}^0(M,E)$ and $K$ a compact operator.
\end{rem}

\subsection{First reduction}\label{ss:first-reduction}
The factorization of $P_{\gG_+}(A)=\frac {-1}{2\pi i} A\Phi(A)$ in
Equation \eqref{def} of Definition \ref{d:sectorial-projection}
permits a first reduction of our problem.

\begin{lemma}\label{l:first-reduction}
Suppose that the map
\[
\Phi:\ \Ell_{\gG_+}^m(M,E)\,\ni\, A\ \mapsto\ \int_{\gG_+}\gl\ii\,
(A-\gl)\ii\,d\gl\, \in\, \cB(H^s,H^{s+m})
\]
is continuous and that $\|\cdot\|_{s, s+m}$ is a continuous
(semi-)norm with respect to $\cT$. Then our claim
\eqref{e:main-mapping} holds.
\end{lemma}

\begin{proof}
Given $A\in\Ell_{\gG_+}^m(M,E)$. Then there is a neighborhood $U$ of
$A$ such that $\|\cdot\|_{s+m,s}$ is bounded on $U$. Hence we reach
the conclusion from
\begin{multline*}
\|P_{\gG_+}(A)- P_{\gG_+}(B)\|_{s,s}\\
\le \|A- B\|_{s+m,s}\, \|\Phi(A)\|_{s,s+m}
+\|B\|_{s+m,s}\,\|\Phi(A)- \Phi(B)\|_{s,s+m}\,.\qedhere
\end{multline*}
\end{proof}

This Lemma reduces the problem to the task of considering
$\int\gl\ii\, (A-\gl)\ii\,d\gl$, which is more convenient.


\section{Local considerations}\label{s:local}
In our Definition \ref{d:topology}, we specified  our topology
$\cT$, see also Remark \ref{r:goal}(b). Now we shall successively
identify the corresponding (semi-)norms on $\pdo_{\SC}^m$ which
ensure that $P_{\gG_+}$ is continuous.

\subsection{Cut-off symbols}\label{ss:cut-off}
In the Appendix, we explain why we cannot deform and extend 
the symbol $a_m(x,\xi)$ in a suitable way.
However, we can easily
deform and extend $\bigl(a_m(x,\xi)-\gl\bigr)\ii$ in the usual way
as a {\em smoothed resolvent symbol} \cite[Sections 11.3-11.4]{Shu:POS} (similarly,
e.g., {\sc Bilyj, Schrohe}, and {\sc Seiler} in the recent
\cite[Definition 2.5]{BilSchSei:CHP}).
Recall that we denote the principal symbol of $A$ by $a_m(x,\xi)$
and that we have assumed that
\begin{equation}\label{e:ray-assumption}
\spec a_m(x,\xi)\cap L_{\ga_j}=\emptyset \qquad\text{ for
$(x,\xi)\in T^*M, \xi\neq 0, j=1,2$}.
\end{equation}
Thus, there is a constant $\rho>0$ such that $a_m(x,\xi)-\gl$ is
invertible for $(x,\xi)\in T^*M, |\xi|\ge \rho$ and $\gl\in\gG_+$\,.
Hence, for any cut-off function $\psi\in\Ci(\R^n)$ with
\begin{equation}\label{e:cut-off}
\psi(\xi)=\begin{cases} 0, & \text{ for $|\xi|\le\rho$},\\
1, & \text{ for $|\xi|\gg 0$},
\end{cases}
\end{equation}
(that is, the function $1-\psi$ is compactly supported) and for each
$\gl\in\gG_+$ the symbol
\begin{equation}\label{e:cut-off-symbol}
(x,\xi)\mapsto \psi(\xi)(a_m(x,\xi)-\gl)\ii
\end{equation}
is a classical symbol of order $-m$.

\subsection{Symbol estimates and semi-norms}\label{ss:symbol-estimates}
From now on we shall switch forward and backward between arguing
locally (in the open domain $U\< \R^n$) and globally (on $M$). With
the preceding symbol $a_m$ and cut-off function $\psi$, we shall
write
\begin{equation}\label{e:lambda-dependent-symbol}
r^\psi(x,\xi,\gl):=\psi(\xi)(a_m(x,\xi)-\gl)\ii\,.
\end{equation}
For fixed $\gl$ we have $r^\psi(\cdot,\cdot,\gl)\in
\CS^{-m}(U\times\R^n,E)$. Considered as a $\gl$-dependent symbol, it
does not necessarily belong to the usual parameter dependent
calculus. Actually, the cut-off $\psi$ prevents this.

However, we have the following symbol estimates, which are uniform
in $\gl\in\gG_+$\/:
\begin{equation}\label{e:parameter-symbol-estimeates}
\begin{split}
\bigl|\partial_x^{\ga}
\partial_\xi^{\gb}\,&r^\psi(x,\xi,\gl)\bigr|\\
&\le \begin{cases} C_{0,0}(1+ |\xi| + |\gl|^{1/m})^{-m}\,, &
\ga=\gb=0,\\
C_{\ga,\gb}(1+ |\xi|)^{m-|\gb|}\, (1 + |\xi| + |\gl|^{1/m})^{-2m}\,,
& (\ga,\gb)\neq (0,0),
\end{cases}\\
&\le C_{\ga,\gb}(1+ |\xi|)^{-|\gb|}\, (1 + |\xi| +
|\gl|^{1/m})^{-m}\,.
\end{split}
\end{equation}
The proof is an exercise in induction and Leibniz rule.

What is important is that the best constants $C_{\ga,\gb}$ in
\eqref{e:parameter-symbol-estimeates}, as functions of $a_m$\,, are
continuous semi-norms on the space $C^k(S^*M;\End(\pi^*E))$ of sections
with $k:=|\ga| + |\gb|$. In particular, they are
continuous semi-norms on the space $\Ci(S^*M;\End(\pi^*E))$.

As a consequence, we have the following: for fixed $k$
and fixed symbol $a_m\in\Ci(S^*M;\End(\pi^*E))$ there is an open
neighborhood $U$ of $a_m$ such that $C_{\ga,\gb}, |\ga| + |\gb|\le
k$, are bounded on $U$ and such that each $b_m\in U$ is
``invertible" on $\gG_+$\,, that is, it satisfies the same
$\Ell_{\gG_+}$--conditions as $a_m$\,.

{\em Note}. We have to fix $k$ and cannot bound infinitely many
semi-norms simultaneously: the intersection of infinitely many open
$U_{\ga,\gb}$ might be non-open.

\subsection{A first approximation} 
The symbolic calculus yields the following first approximation
result.

\begin{prop}\label{p:first-approximation}
{\rm (a)} For $a_m$ and $r^\psi$ as above, the operator
\begin{equation}\label{e:first-approximation}
\Phi_0(a_m) := \int_{\gG_+}\gl\ii\,\Op\bigl(r^\psi(\cdot,\cdot,\gl)\bigr)\,d\gl
\end{equation}
belongs to the class $\CL^{-m}(U,E)$.

\noi {\rm (b)} $\Phi_0 \circ \sigma_m: \Ell^m_{\gG_+}(M,E) \to
\cB(H^s,H^{s+m})$ is continuous with regard to $\cT$.
\end{prop}

\begin{proof} For (a) we see that
\[
\psi(\xi)\int_{\gG_+}\gl\ii\, (a_m(x,\xi)-\gl)\ii\, d\gl =
\int_{\gG_+}\gl\ii\, r^\psi(x,\xi,\gl)\, d\gl \] is homogeneous of
degree $-m$ outside a compact set, and smooth otherwise. Recall that
principal symbols are determined by their values in $\{(x,\xi)\in
T^*M\mid |\xi|\ge C\}$ where $C$ is any positive constant. That
proves (a).

\noi For (b) we have that in the topology $\mathcal T$,
$$\sigma_m:{\rm L}^m_{{\rm pc}}(M,E)\to C^{\infty}(S^*M, {\rm
End}(\pi^*E))$$ is continuous. We denote the space of symbols
analogue of the operator space $\Ell_{\gG_+}^m$ by
$\Ci_{\gG_+}(S^*M, \End(\pi^*E))$. Certainly,
\[
\begin{matrix}
\Ci_{\gG_+}(S^*M, \End(\pi^*E)) & \too & \CS^{-m}(T^*M, \End(\pi^*E))\\
a_m & \longmapsto & \int_{\gG_+}\gl\ii\,
\psi(\xi)\,(a_m(x,\xi)-\gl)\ii\, d\gl
\end{matrix}
\]
and
\[
\Op: \, \CS^{-m}(T^*M, \End(\pi^*E)) \too \cB(H^s,H^{s+m})
\]
are continuous. That proves (b).
\end{proof}

\section{A technical lemma and key estimates}\label{s:mtl}

\subsection{A technical lemma}
In this subsection, we shall give a technical lemma which is crucial
in the proof of our main theorem. It is a variant form (with a
parameter) of the composition of pseudo-differential operators. As a
service to the reader, we give a detailed proof of this lemma in the
Appendix \ref{s:appendixA}. Our claims and arguments are local for a
fixed open coordinate patch $U\subset \R^n$\/.

\begin{dfn}\label{d:symbol-K}
For a compact subset $K\<U$ we denote by $\sym_K^m(U\times
\R^n)\<\sym^m(U\times\R^n)$ those $a\in \sym^m(U\times\R^n)$ such
that $a(x,\xi)\ne 0$ implies $x\in K$. $\Csym^m_K(U\times \R^n)$ is
defined accordingly. A typical example is
$a(x,\xi)=\theta(x)b(x,\xi)$ with $b\in \sym^m(U\times\R^n)$ and a
cut-off function $\theta\in C_c^{\infty}(U)$.
  \end{dfn}

Clearly, the preceding definitions carry over to matrix valued
symbols and to globally defined symbols with values in bundle
endomorphisms.

\begin{lemma}[Technical Lemma]\label{mtl}
Let $m>0,\,0\leq r\leq m$. Let $f,g\in\Ci(U\times\R^n\times\gG_+)$
such that for $\gl\in\gG_+$
$$f(\cdot,\cdot,\gl)\in \sym^r_K(U\times\R^n),\qquad
g(\cdot,\cdot,\gl)\in \sym^{-m}_K(U\times\R^n).$$ Assume that
\begin{multline}\label{f:e}
|\partial_x^{\alpha}\partial_{\xi}^{\beta}f(x,\xi,\lambda)|\\
\leq\begin{cases}C_{0,0}(f)(1+|\xi|+|\lambda|^{1/m})^{r}, &\alpha=\beta=0,\\
C_{\alpha,\beta}(f)(1+|\xi|)^{m-|\beta|}(1+|\xi|+|\lambda|^{1/m})^{r-m},
& |\alpha|+|\beta|>0,\end{cases}
\end{multline}
and
\begin{equation}\label{g:e}
|\partial_x^{\alpha}\partial_{\xi}^{\beta}g(x,\xi,\lambda)|
\leq \wt C_{\alpha,\beta}(g)(1+|\xi|)^{-|\beta|}(1+|\xi|+|\lambda|^{1/m})^{-m}\/,
\end{equation}
where $C_{\cdot,\cdot}(\cdot), \wt C_{\cdot,\cdot}(\cdot)$ are constants depending on certain datas in the dots' positions. Set $C_N(f)=\sum_{|\alpha|,|\beta|\leq N}C_{\alpha,\beta}(f)$ and
$\wt C_N(g)=\sum_{|\alpha|,|\beta|\leq N}\wt C_{\alpha,\beta}(g)$.
Then for $s\in \R$, there is an $N(s)\in \N$ and $C>0$ such that
\begin{multline*}
\left\|{\rm Op}(g(\cdot,\cdot,\lambda)){\rm Op}(f(\cdot,\cdot,\lambda))-{\rm Op}(g f(\cdot,\cdot,\lambda))\right\|_{s,s+m-r}\\
\leq C C_{N(s)}(f)\wt C_{N(s)}(g)|\lambda|^{-\min(\frac{1}{m},1)}.
\end{multline*}
\end{lemma}

\begin{rem}
We should notice that $C_N(\cdot)$ and $\wt{C}_N(\cdot)$ are
semi-norms if we choose the smallest constants
$C_{\alpha,\beta}(\cdot)$ and $\wt C_{\alpha,\beta}(\cdot)$ in
(\ref{f:e}), (\ref{g:e}). Moreover, $C_N(f)$ and $\wt{C}_N(g)$ are
dominated by the finitely many constants $C_{\alpha,\beta}(f)$ and
$\wt C_{\alpha,\beta}(g)$, $|\alpha|,\,|\beta|\leq N$, respectively.
\end{rem}

Now we give some additional examples.
\begin{exmpl}
 $g(x,\xi,\lambda):=\psi(\xi)(a_m(x,\xi)-\lambda)\ii$ satisfies (\ref{g:e}). See (\ref{e:parameter-symbol-estimeates}).
\end{exmpl}
\begin{exmpl}
$f(x,\xi,\lambda):=a(x,\xi)-\lambda$ satisfies (\ref{f:e}) with
$r=m$. If $b\in \Csym_K^m(U\times \R^n)$ is a symbol of order $m$,
then $f(x,\xi,\lambda):=\psi(\xi)(a_m(x,\xi)-\lambda)\ii
b(x,\xi)=r^{\psi}(x,\xi)b(x,\xi)$ also satisfies (\ref{f:e}) with
$r=0$. Note that in this case $$\sum_{|\alpha|,|\beta|\leq
N}C_{\alpha,\beta}(f)\leq \left(\sum_{|\alpha|,|\beta|\leq
N}C_{\alpha,\beta}(r^{\psi})\right)\left(\sum_{|\alpha|,|\beta|\leq
N}C_{\alpha,\beta}(b)\right).$$ Here $C_{\alpha,\beta}(b)$ denotes
the best constant in the symbol estimate for
$\partial_x^{\alpha}\partial_{\xi}^{\beta}b(x,\xi)$ and
$C_{\alpha,\beta}(r^{\psi})$ is of similar meaning.
\end{exmpl}

\begin{rem}
Note that in the examples above, $C_{\alpha,\beta}(f)$ and
$C_{\alpha,\beta}(g)$ are bounded by a $C^k$-norm on $a_m$ (and
$b_m$ in the preceding example) for sufficiently large $k$.
\end{rem}


\subsection{Key estimates}\label{s:key-estimates} Before proving the
main result of this note, we give some more estimates.
\begin{lemma}\label{norm-estimate1}
Given $A\in {\rm Ell}^{m}_{\Gamma_+}(M,E)$. Then for $s\in
\mathbb{R}$, $0\leq p\leq m$, and all $\lambda\in \Gamma_+$ we have
\begin{equation}
\|{\rm Op}(\psi(a_m-\lambda)^{-1})\|_{s,s+p}\leq
C_s(A)|\lambda|^{-1+\frac{p}{m}}.
\end{equation}
Furthermore, to $s$ there is $N_s\in\N$ such that $C_s(A)$ is bounded by
the $C^{N_s}$-norm of $a_m$ on $S^*M$.
\end{lemma}

In other words, to $A$ there is an open neighborhood $U$ of $a_m$
(in the $C^{N_s}$-topology) such that the map $B\mapsto C_s(B)$ is
bounded on the open set $\sigma^{-1}_m(U)$.

\begin{proof} Use the standard method of estimating norms of
pseudo-differential operators as in {\sc Seeley} \cite[Lemma
2]{See:CPE}. Of course it also follows from the method presented in
the preceding section.
\end{proof}

\begin{lemma}\label{norm-estimate2}
Given $A\in {\rm Ell}^{m}_{\Gamma_+}(M,E)$. Then for $s\in
\mathbb{R}$ and all $\lambda\in \Gamma_+$
\begin{equation}
\|{\rm Op}(\psi(a_m-\lambda)^{-1})-(A-\lambda)^{-1}\|_{s,s+m}\leq
C_s(A)|\lambda|^{-\min(\frac{1}{m},1)}.
\end{equation}
$C_s(A)$ has the same property as in \textup{Lemma \ref{norm-estimate1}}.
\end{lemma}

\begin{proof} Put $A={\rm Op}(a)$ for the complete symbol $a$. Write
$a=a_m+a_{m-1}$. Then we have
\begin{multline*}
(A-\lambda)({\rm Op}(\psi(a_m-\lambda)^{-1}-(A-\lambda)^{-1}))
={\rm Op}(a_m-\lambda){\rm Op}(\psi(a_m-\lambda)^{-1})\\
-{\rm Op}(\psi)-{\rm Op}(1-\psi)+{\rm Op}(a_{m-1}){\rm
Op}(\psi(a_m-\lambda)^{-1}).
\end{multline*}
Note that $(A-\lambda)(A-\lambda)^{-1}=I={\rm Op}(1)$ and ${\rm
Op}(1-\psi)$ is a smoothing operator (because $1-\psi$ is compactly
supported). Hence
\begin{eqnarray*}
&&\|{\rm Op}(\psi(a_m-\lambda)^{-1})-(A-\lambda)^{-1}\|_{s,s+m}\\
&\leq& \|(A-\lambda)^{-1}\|_{s,s+m}\|{\rm Op}(a_m-\lambda){\rm
Op}(\psi(a_m-\lambda)^{-1})-{\rm Op}(\psi)\|_{s,s} \\
&&+\|(A-\lambda)^{-1}\|_{s+m,s+m}\|{\rm Op}(1-\psi)\|_{s,s+m}\\
&&+\|(A-\lambda)^{-1}\|_{s,s+m}\|{\rm
Op}(a_{m-1})\|_{s+m-1,s}\|{\rm
Op}(\psi(a_m-\lambda)^{-1})\|_{s,s+m-1}\\
&\leq& C_s(A)|\lambda|^{-\min(\frac{1}{m},1)}
\end{eqnarray*}
by the technical lemma \ref{mtl}, applied to $f=a_m-\lambda$,
$g=\psi(a_m-\lambda)^{-1}$ and Lemma \ref{norm-estimate1}. The local
boundedness claim on $C_s(A)$ also follows from this lemma.
\end{proof}

\section{Applications}\label{s:applications}
As an application of the technical lemma and the preceding
estimates, we prove that the sectorial projections depend
continuously on the underlying operators in the topology
$\mathcal{T}$ to be fixed below. We shall explain in detail how that
perturbation result yields the continuous variation of the
\Calderon\ projection (and hence of the Cauchy data spaces) of
arbitrary linear elliptic differential operators of first order on
smooth compact manifolds with boundary under the assumption of the
inner unique continuation property.

\subsection{The operator type of the sectorial projection}\label{ss:operator-type}
As explained in the Introduction, \textsc{Wodzicki}'s equation
\eqref{e:wodzicki} yields at once that the sectorial projection is a
pseudo-differential operator of order 0, at least for classical
pseudo-differential operators.

Our estimates are designed for the perturbation problem and do not
give such a sharp result. All we can derive immediately is that the
operator $P_{\gG_+}(A)$ is bounded $H^s(M;E)\to H^s(M;E)$ for all
$s\in\R$. Note that we do not require that $A$ is classical. We only
assume that $A$ is  principally classical.

As usually, we argue locally. By Proposition \ref{p:first-approximation}a,
$\Phi_0(a_m) \in \CL^{-m}(U,E)$. Furthermore, by Lemma \ref{norm-estimate2}
\begin{equation*}
\|{\rm Op}\bigl(\psi(a_m-\lambda)^{-1}\bigr)-(A-\lambda)^{-1}\|_{s,s+m}\leq
C_s(A)|\lambda|^{-\min(\frac{1}{m},1)}
\end{equation*}
for $\gl\in\gG_+$\/, thus by Def. \plref{d:sectorial-projection} and \eqref{e:first-approximation}
\begin{equation*}
\left\|P_{\gG_+}(A)-A\Phi_0(a_m)\right\|_{s,s} \leq
C_s(A)\int_{\gG_+}|\lambda|^{-1-\min(\frac{1}{m},1)}|d\lambda|,
\end{equation*}
and the claim follows.

\subsection{Proof of Theorem \ref{mt}}\label{proof:mr}

From now on we equip ${\rm Ell}^{m}_{\Gamma_+}(M,E)$ with the
topology $\mathcal{T}$. Let $A\in {\rm Ell}^{m}_{\Gamma_+}(M,E)$ and
$\Delta A$ be in a neighborhood of $0$. Since
$\Phi(A)=\int_{\Gamma_+}\lambda^{-1}(A-\lambda)^{-1}d\lambda$ and
$A\rightarrow \Phi_0(A)$ is continuous, it is sufficient to prove an
estimate, uniformly on $\Gamma_+$, of the form
\begin{multline}\label{aim}
\|(A+\Delta A-\lambda)^{-1}-(A-\lambda)^{-1}
-{\rm Op}\bigl(\psi((a_m+\Delta
a_m-\lambda)^{-1}-(a_m-\lambda)^{-1})\bigr)\|_{s,s+m}\\
\leq C_s(A,\Delta A)|\lambda|^{-\min(\frac{1}{m},1)}\/,
\end{multline}
such that the following holds:
given $\epsilon>0$, there is a neighborhood $U$ of $0$ (in the
locally convex topology $\mathcal{T}$) such that for all $\Delta
A\in U$, $C_s(A,\Delta A)<\epsilon$.

To prove \eqref{aim}, we make an elementary algebraic re-ordering of
the left side of \eqref{aim} into five summands and invoke the
triangle inequality successively:
\begin{align*}
\text{\eqref{aim}},&\text{ left side}
=\|(A-\gl)\ii(-\Delta A)(A+\Delta A-\gl)\ii\\
&\qquad -\Op\Bigl(\psi(a_m-\gl)\ii\,(-\Delta a_m)(a_m+\Delta
a_m-\gl)\ii\Bigr)\|_{s,s+m}\\
&\le \|(A-\gl)\ii-\Op(\chi(a_m-\gl)\ii)\|_{s,s+m}\\
&\qquad \cdot \|\Delta A(A+\Delta A-\gl)\ii\|_{s,s}\\
&\quad + \|\Op(\chi(a_m-\gl)\ii)\|_{s,s+m}\ \cdot\\
&\qquad \|\Delta A-\Op(\chi_1\Delta a_m)\|_{s+m-1,s}\, \|(A+\Delta A
-\gl)\ii\|_{s,s+m-1}\\
&\quad + \|\Op(\chi(a_m-\gl)\ii)\Op(\chi_1\Delta a_m)\|_{s+m,s+m}\ \cdot\\
&\qquad \|A+\Delta A -\gl)\ii -\Op(\chi_2(a_m+\Delta a_m-\gl)\ii)\|_{s,s+m}\\
&\quad + \|\Bigl\{\Op\bigl(\chi(a_m-\gl)\ii\bigr)\Op(\chi_1\Delta
a_m) -\Op\bigl(\chi(a_m-\gl)\ii\Delta a_m\bigr)\Bigr\}\\
&\qquad\cdot \Op\bigl(\chi_2(a_m+\Delta a_m-\gl)\ii\bigr)\|_{s,s+m}\\
&\quad + \|\Op\bigl(\chi(a_m-\gl)\ii\Delta
a_m\bigr)\Op\bigl(\chi_2(a_m+\Delta a_m-\gl)\ii\bigr)\\
&\qquad-\Op\bigl(\psi(a_m-\gl)\ii\,\Delta a_m\, (a_m+\Delta
a_m-\gl)\ii\bigr) \|_{s,s+m}\/.
\end{align*}
Here we choose $\chi,\chi_1,\chi_2$ with the same properties like
$\psi$ such that $\chi=\chi\chi_1$ and $\chi\chi_2=\psi$.

Now apply the Technical Lemma \ref{mtl} to the last two
summands and the Lemmata \ref{norm-estimate1} and
\ref{norm-estimate2} to the first three summands, and we are done.

\begin{rem}
We note in passing that in the first three lines of the proof of
Proposition \ref{p:first-approximation} it is decisively used that
the integrand is \emph{homogeneous} in $\lambda$. This is an
obstruction against an immediate generalization of our perturbation
result to general holomomorphic functions of $A$ in an
$H_\infty$--calculus style ({\em cf., e.g.}, \cite{BilSchSei:CHP}).
To explain this a bit more, let $f$ be a function, which is
holomorphic and bounded in a neighborhood of the closure of the
sector encircled by $\Gamma_+$. This is not quite the situation of
the $H_\infty$ calculus, since there are also eigenvalues on the
left of the contour. Nevertheless, one might hope that similarly to
loc. cit. one can prove that $f(A)$ is a bounded operator. Still
there is no immediate analogue of Proposition
\ref{p:first-approximation} in this case and we leave it as an
intriguing open problem whether the Perturbation Theorem \ref{mt}
carries over to $f(A)$ instead of $P_{\Gamma_+}(A)$.
\end{rem}

\subsection{Index correction formulas}
 The sectorial projections are
significant in the celebrated Atiyah-Patodi-Singer Index Theorem. A
common set-up is the following: Let $X$ be a compact smooth
Riemannian manifold with boundary $M$, and $E$ and $F$ be two
Hermitian vector bundles over $X$. Let $D: H^1(X;E)\rightarrow
L^2(X;F)$ be a first order elliptic differential operator and let
$A:H^1(M;E|_M)\to L^2(M;E|_M)$ denote the tangential operator of $D$
on $M$ relative to the fixed metric structures. In the classical
works \cite{APS:SARI, APS:SARII, APS:SARIII}, \textsc{M. Atiyah, V.
Patodi and I. Singer} assumed that $D$ is of Dirac type; all metric
structures near $M$ are product; hence, the coefficients of $D$ in
normal direction close to $M$ are constant and the tangential
operator $A$ is self-adjoint. Imposing a spectral projection
condition $P^+(A)u|_{\partial X}=0$ on the boundary, they proved
that the resulting (densely defined) operator $D_{P^+(A)}$ over $X$
is Fredholm. Furthermore, they gave an index formula, comprising
topological, spectral and differential terms. The arguments of
\cite[p. 95]{APS:SARIII} (worked out in \cite[Theorem
1.4]{SavSte:IDT} and, differently and in detail, in \cite[Theorem
7.6]{LesMosPfl:CCC}) lead to the index correction formula
\begin{equation}
{\rm ind}(D_0)_{P^+(A_0)}-{\rm ind}(D_1)_{P^+(A_1)}={\rm
sf}\{A_t\}_{t\in[0,1]},
\end{equation}
where $\{D_t,\,t\in [0,1]\}$ is a smooth homotopy, and $\{A_t\}$
denotes its corresponding family of tangential operators. It is also
called the {\em Spectral Flow Theorem}. The continuous dependence of
$P^+(A_t)$ on $A_t$ (in the sense that $P^+(A_t)$ has the same jumps
as $1_{(-\gve,\gve)}(A_t)$\/, if $\pm\gve \not\in\spec A_t$) is
important in this theorem. When $A_t$ is self-adjoint, it can be
proved by standard techniques of functional analysis (cf.
\cite[Chapter 17]{BooWoj:EBP} or above Section \ref{sss:bounded} of
this note).

 It is natural to consider a more general case. In
\cite{SavSetSch:IIF}, \textsc{A. Savin, B.-W. Schulze} and {\sc B.
Sternin} gave a similar formula for the case that the tangential
family $A_t$ is non-self-adjoint. It seems very satisfactory that we
now have a proof of the continuous dependence of $P^+(A_t)$ on $A_t$
when $A_t$ has
 no spectral points on the imaginary axis for all $t\in [0,1]$.
\subsection{Continuous dependence of the Calder\'on projection on the
data}\label{ss:calderon}

In  \cite[Sec. 7]{BooLesZhu:CPD} we discussed the problem of
continuous dependence on the input data of the Calder\'on projection
associated to a first order elliptic differential operator on a
compact manifold with boundary. On the one hand this can be viewed
as the rather classical problem of showing that the space of
solutions of a PDE (here the equation $Du=0$ in the interior)
depends continuously on the data. The question arises naturally in
connection with the Spectral Flow Theorem of {\em symplectic
geometry} and has been proved in various special cases (see e.g.
\cite{CLM:SEOI,CLM:SEOII,Nic:MIS,BooFur:MIF,KirLes:EIM,HKL:CPH}).

Since our main motivation for writing the current note comes from
this problem (see the recent \cite{Boo:BFA}), let us briefly describe
the set-up and the main result of \cite[Sec. 7]{BooLesZhu:CPD} as
well as the improvement provided by Theorem \plref{mt}.

Let $X$ be a compact connected manifold with boundary $M$ and
$E,F$ vector bundles over $X$. We fix a Riemannian metric and
Hermitian metrics on the vector bundles to have Hilbert space
structures on the sections of $E,F$. We choose the metrics in such a
way that all structures are product in a collar neighborhood
$U=[0,\eps)  \times M$ of the boundary. We emphasize that this
is not a loss of generality since we will consider \emph{variable}
coefficient differential operators, see the detailed discussion in
\cite[Sec. 2.1]{BooLesZhu:CPD}.

For a first order elliptic differential operator $D\in\Diff^1(X;E,F)$
we write (cf. \cite[(2.18), (2.19), (5.11)-(5.16)]{BooLesZhu:CPD})
in the collar $U$:
\begin{equation}\label{eq:operator-collar}
\begin{split}
    D   &= J_x\Bigl(\frac{d}{dx} + B_x\Bigr)\\
        &=:J_0\Bigl(\frac{d}{dx} + B_0\Bigr)+C_1 x - J_0'\\
    D^t &= \Bigl(-\frac{d}{dx} + B_0^t\Bigr)J_0^t+\widetilde C_1 x.
\end{split}
\end{equation}
where $J_x\in \operatorname{Hom}(E_M,F_M), 0\leq x\leq\varepsilon,$
is a smooth family of bundle homomorphisms and $(B_x)_{0\leq x\leq\varepsilon}$ is a smooth family
of first order elliptic differential operators between sections of
$E_M$.  $C_1, \widetilde C_1$ are first order differential operators and, by slight
abuse of notation, $x$ will also denote the operator of multiplication by the function $x\mapsto x$.
We note that $x$ is intentionally on the right of $C_1,\widetilde C_1$.
$D^t$ denotes the formal adjoint of $D$ with respect to the $L^2$-structure.
We consider $J_x, B_x, C_1, \widetilde C_1$ as functions of $D$.

Next we denote by $\cE(X;E,F)$ the set of pairs
$(D,T)\in\Diff^1(X;E,F)\times \Diff^0(M;E_M,F_M)$
such that
\begin{thmenum}
\item $D$ is elliptic
\item $J_0^tT$ is positive definite (in particular self-adjoint) and
the commutator $[J_0^tT,B_0]$ is a differential operator of order $0$.
\end{thmenum}

Recall from \cite[Sec. 4]{BooLesZhu:CPD}
the invertible double construction associated to a pair
$(D,T)\in \cE(X;E,F)$: Put
\begin{equation}
\widetilde D:=D\oplus (-D^t),
\end{equation}
acting on sections of $E\oplus F$, and impose the
boundary condition
\begin{equation}
\binom{f_+}{f_-}\in\operatorname{dom}(\widetilde D_T):\Leftrightarrow  {f_-}|_{\partial X}=T {f_+}|_{\partial X}
\Leftrightarrow ({f_+}|_{\pl X}, {f_-}|_{\pl X})\in\ker \begin{pmatrix}
-T &\Id \end{pmatrix}.
\end{equation}

\cite[Theorem 4.7]{BooLesZhu:CPD} states that $\widetilde D_T$
is a realization of a local elliptic boundary value problem (in the classical
{\v S}apiro-Lopatinski{\v i} sense) and that the kernel and cokernel
of $\widetilde D_T$ are isomorphic to the direct sum of the spaces of ghost
solutions $Z_0(D)=\bigl\{ u\in L^2(X,E)\,|\, Du=0, u|_{\partial
X}=0\bigr\}$ and $Z_0(D^t)$ for $D$ and $D^t$. In particular if $D$ and $D^t$
satisfy the weak inner \UCP\ (i.e. $Z_0(D)=0=Z_0(D^t)$) then $\widetilde D_T$
is indeed invertible. This \emph{canonical} invertible double
construction is the natural generalization of the geometric invertible
double construction for Dirac type operators in the product situation
(cf. e.g. \cite{BooWoj:EBP}) to general first order elliptic differential
operators.

Furthermore, in \cite[Sec. 5]{BooLesZhu:CPD} we showed that from $\widetilde D_T$
one obtains a projection (the Calder\'on projection) onto the Cauchy data space $N^0(D):= \bigsetdef{u|_{\partial X}\in
L^2(\partial X; E|_{\pl X})}{ Du=0}$ of $D$ by the formula \cite[(5.31)]{BooLesZhu:CPD}
\begin{equation}
      C_+(D,T)= \Bigl(P_+ - \varrho_+ \widetilde G S(D,T)\Bigr)(P_++P_-^*)\ii.
\end{equation}
Note that the range of $C_+(D,T)$ equals $N^0(D)$ and is independent of $T$.
However, $C_+(D,T)$ is in general not an orthogonal projection and may depend on $T$.
$C_+(D,J_0^t)$ is indeed the orthogonal projection onto $N^0(D)$.

Now denote by $\cE_{\UCP}(X;E,F)$ the set of $(D,T)\in\cE(X;E,F)$ such that
$D$ and $D^t$ satisfy weak inner \UCP.

Let $\cV(X;E,F)$ be the linear subspace of
$\Diff^1(X;E,F)\times \Diff^0(M;E_M,F_M)$
consisting of those $(D,T)$ such that $[B_0^t,J_0^t T]$ is of
order $0$; and introduce the following two norms on $\cV(X;E,F)$:
\begin{align}\label{eq:weak-norm}
   N_0(D,T)&:= \|D\|_{1,0}+\|D^t\|_{1,0}+\|T\|_{\half,\half}\,,\\
\intertext{and}
    N_1(D,T)&:=\|B_0\|_{1,0}+\|B_0^t\|_{1,0}+\|[B_0^t,J_0^tT]\|_0
                       +\|T\|_0 \label{eq:strong-norm}\\
                   &\quad +\|J_0\|_0+ \|C_1\|_{1,0}+\|\widetilde C_1\|_{1,0}.\nonumber
\end{align}
Compared to \cite[(7.1), (7.2)]{BooLesZhu:CPD} we have omitted a few redundant terms.
We obtain a metric on $\cV(X;E,F)$ and hence on its subsets by putting
\begin{equation}\label{eq:strong-metric}
\dstr((D,T),(D',T')):=N_0(D-D',T-T')
    +N_1(D-D',T-T').
\end{equation}

Finally, let $\Gamma$ be a contour as in \eqref{curve} and let $\cE_{\UCP,\Gamma}(X;E,F)$
be the set of those $(D,T)\in \cE_{\UCP}(X;E,F)$ such that the leading symbol of $B_0$ has
no eigenvalues on the two rays $L_{\ga_j}$ of $\Gamma$ and $B_0$ no eigenvalues on $\Gamma$.

\cite[Theorem 7.2 (b)]{BooLesZhu:CPD} can now be phrased as follows:

\begin{theorem}\label{t:CalderonContinuous1} Let $s\in [-1/2,1/2]$ and let $\cT_\Gamma$ be the coarsest topology
on $\cE_{\UCP,\Gamma}(X;E,F)$ such that
\begin{thmenum}
\item $\dstr$ is continuous on $\cE_{\UCP,\Gamma}(X;E,F)\times \cE_{\UCP,\Gamma}(X;E,F)$,
\item $(D,T)\mapsto P_{\Gamma}(B_0)\in\cB(H^s(M;E_M))$ is continuous.
\end{thmenum}
Then the map $\cE_{\UCP,\Gamma}(X;E,F)\ni (D,T)\mapsto C_+(D,T)\in\cB(H^s(M;E_M))$
is continuous.
\end{theorem}

As pointed out in \cite[Remark 7.3]{BooLesZhu:CPD} the obvious weakness of this result
is that the continuous dependence of $P_\Gamma(B_0)$ has to be assumed.

Combining Theorem \plref{t:CalderonContinuous1} with Theorem \ref{mt} we obtain
a much more satisfactory formulation of the continuous dependence
of $C_+(D,T)$ without reference to a positive sectorial projection:

\begin{theorem}\label{t:CalderonContinuous2} Let $s\in [-1/2,1/2]$ and let $\cT$ be the coarsest topology
on $\cE_{\UCP}(X;E,F)$ such that
\begin{thmenum}
\item $\dstr$ is continuous on $\cE_{\UCP}(X;E,F)\times \cE_{\UCP}(X;E,F)$,
\item The leading symbol map for the tangential operator
$\sigma:\cE_{\UCP}(X;E,F)\mapsto \Gamma^\infty(S^*M,\pi^*(\End E_M)),
(D,T)\mapsto \sigma_1(B_0)$ ($\pi:S^*M\to M$ the projection map)
is continuous when $\Gamma^\infty(S^*M,\pi^*(\End(E_M))$ is equipped
with the $C^\infty$--topology.
\end{thmenum}
Then $C_+:(\cE_{\UCP}(X;E,F),\cT)\longrightarrow \cB(H^s(M,E_M))$
is continuous.
\end{theorem}

\begin{rem}
\noindent \textup{1. } The restrictions $s\in[-1/2,1/2]$ in Theorems \plref{t:CalderonContinuous1}
and \plref{t:CalderonContinuous2} are not serious. For other values of $s$ the metric
$\dstr$ has to be modified in a fairly straightforward way. 

\noindent \textup{2. }
Let $\cT^\infty$ be the coarsest topology on $\cE_{\UCP}(X;E,F)$ such that in each
chart the coefficient functions of a coordinate representation of $D\in\cE_{\UCP}(X;E,F)$
vary continuously in the $C^\infty$--topology (cf. e.g. \cite[Theorem 3.16]{HKL:CPH}).

Then it is a routine matter to check that $\cT^\infty$ is finer than the topology $\cT$ of
Theorem \plref{t:CalderonContinuous2}. In fact it is fine enough to guarantee
the continuity $C_+:(\cE_{\UCP}(X;E,F),\cT^\infty)\longrightarrow \cB(H^s(M,E_M))$
for \emph{all} real $s$ (cf. item 1. of this Remark).

This version of the continuous dependence of $C_+$, although strictly speaking
somewhat weaker than Theorem \plref{t:CalderonContinuous2}, is probably the most satisfactory
way of summarizing its content.

\end{rem}


\subsection*{Acknowledgment}
We are indebted to Prof. \textsc{Kenro Furutani} (Tokyo) for
initiating this note some years ago, by asking us about the
continuous variation of Cauchy data spaces, and to Prof.
\textsc{Gerd Grubb} (Copenhagen) and Prof. \textsc{Elmar Schrohe}
(Hannover) for various suggestions to this work. In particular, we 
wish to thank the referee for his or her criticism and
suggestions that helped to condensate our arguments and, hopefully, lead to an easier
readable note.

\appendix

\section{}\label{s:appendixA}
In this appendix, we provide the details of the proof of our
technical lemma \ref{mtl}. We first translate the wanted estimates
into statements about integral operators.

\subsection{$L^2$-estimates for integral operators and other estimates}\label{ss:l2-estimates}
We recall the well-known and very useful {\sc Schur}'s Test for
integral operators (see, e.g., {\sc Halmos} and {\sc Sunder}
\cite[Theorem 5.2]{HalSun:BIO}):

\begin{lemma}[Schur's Test] Let $K$ be an integral operator with measurable
kernel $k:\R^n\times\R^n\to\C$. Assume that
\[
\sup_{x\in\R^n}\int_{\R^n} |k(x,y)|dy \le C_1 <+\infty \tand
\sup_{y\in\R^n}\int_{\R^n} |k(x,y)|dx \le C_2 <+\infty.
\]
Then $K$ is bounded $L^2(\R^n)\to L^2(\R^n)$ and $\|K\|_{L^2\to L^2}
\le \sqrt{C_1C_2}$.

In particular, if for some $p>n$
\[
|k(x,y)|\le C_3(1+|x-y|)^{-p}\/,
\]
then the criterion is fulfilled with
\[
C_1=C_2=C_3\int_{\R^n}(1+|\xi|)^{-p}d\xi.
\]
\end{lemma}

Now fix $U\< \R^n$ open, $K\< U$ compact and
$a\in\sym_K^m(U\times\R^n)$. Then {\sc Schur}'s Test yields an
effective estimate for $\|\Op(a)\|_{s,s-m}$\/.

To explain that, we introduce some notations. For the Fourier
transform, we shall follow {\sc H{\"o}rmander}'s convention
\begin{multline*}
\bigl(\cF f\bigr)(\xi):=\int_{\R^n} e^{-i\lla x,\xi\rra} f(x)
dx,\quad
\bigl(\cF\ii u\bigr)(x):=(2\pi)^{-n}\int_{\R^n} e^{i\lla x,\xi\rra} u(\xi) d\xi\\
\dbar \xi:=(2\pi)^{-n}d\xi.
\end{multline*}
Then we have
\begin{multline*}
\bigl(\cF \Op(a)u\bigr)(\eta):=\int_{\R^n} e^{-i\lla
x,\eta\rra}\bigl(\Op(a)u\bigr)(x) dx\\ =\int_{\R^n}\Bigl[\int_U
e^{i\lla x,\xi-\eta\rra}a(x,\xi)\dbar x\Bigr] \wh u(\xi) d\xi.
\end{multline*}
We set $q_a(\xi-\eta,\xi):=[\cdots]$ in the preceding formula and
define

\begin{dfn}\label{d:qa}
For $a\in\sym_K^m(U\times\R^n)$, we set
\[
q_a(\zeta,\xi):=\bigl(\cF\ii_{x\to\zeta} a(x,\xi)\bigr)(\zeta) =
\int_{K} e^{i\lla\zeta,x\rra} a(x,\xi)\dbar x.
\]
\end{dfn}
Since $a(x,\xi)$ is nonzero at most if $x\in K$ the integral
certainly exists. The method we are going to employ is adapted from
\cite[Lemma 1.2.1]{Gil:ITH}. Lemma 1.2.1 (b) of loc. cit. shows that
$q_a(\zeta,\xi)$ decays to arbitrarily high powers (reproved below)
in $\zeta$ (as $\zeta\to\infty$) and is polynomially bounded in
$\xi$. Hence all integrals below converge in the usual sense.

Consequently, the kernel of the integral operator $\cF\Op(a)\cF\ii$ is given
by $ k_a(\tau,\xi):=q_a(\xi-\tau,\xi). $ To estimate the operator
norm $\|\cdot\|_{s,s-m}$ of $\Op(a)$ it suffices therefore to
estimate the norm of the operator $\cF\Op(a)\cF\ii$ as a
map from the weighted $L^2$-space $L^2(\R^n,(1+\|\xi\|^2)^s)$ into
$L^2(\R^n,(1+\|\xi\|^2)^{s-m})$. By {\sc Schur}'s test an estimate
of the form
\begin{equation}\label{e:schur-est}
\bigl|(1+|\tau|)^{s-m} k_a(\tau,\xi)(1+|\xi|)^{-s}\bigr|\le
C(a)C(p)(1+|\tau-\xi|)^{-p} \qquad\text{ for some $p>n$}
\end{equation}
implies
\[
\|\Op(a)\|_{s,s-m}\le C(a)\wt C(p) \qquad\text{ with $\wt
C(p):=C(p)\int_{\R^n}(1+|x|)^{-p}dx$}.
\]

\begin{proof}[Proof of the Main Technical Lemma]
Let $f(\cdot,\cdot,\gl)\in \sym^r_K(U\times \R^n)$,
$g(\cdot,\cdot,\gl)\in \sym^m_K(U\times \R^n)$ satisfying
(\ref{f:e}), (\ref{g:e}) be given. In the sequel we will suppress
the argument $\lambda$ from the notation for simplicity. We should
be aware that all expressions will depend on $\lambda$ unless
otherwise stated. The kernel of the operator $\cF\Op(f)\Op(g)\cF\ii$ is given by
\[
k_{f\cdot g}(\tau,\xi)=\int_{\R^n}
k_f(\tau,\eta)k_g(\eta,\xi)d\eta=\int_{\R^n}
q_f(\eta-\tau,\eta)q_g(\xi-\eta,\xi)d\eta.
\]
On the other hand
\begin{align*}
q_{f\cdot g}(\zeta,\xi)&=\int_{\R^n}
e^{i\lla\zeta,x\rra}f(x,\xi)g(x,\xi)\dbar x\\
&=\cF\ii\bigl(f(\cdot,\xi) g(\cdot,\xi)\bigr)(\zeta)\\
&=\int_{\R^n} q_f(\zeta-\eta,\xi)q_g(\eta,\xi)d\eta,
\end{align*}
respectively,
\begin{align*}
k_{f\cdot g}(\tau,\xi)&=q_{f\cdot g}(\xi-\tau,\xi)\\
&=\int_{\R^n} q_f(\xi-\tau-\eta,\xi)q_g(\eta,\xi)d\eta; \quad
\xi-\eta
\rightsquigarrow \eta\\
&= \int_{\R^n} q_f(\eta-\tau,\xi)q_g(\xi-\eta,\xi)d\eta.
\end{align*}
Thus the kernel of $\cF\Bigl\{ \Op(f)\Op(g)-\Op(f\cdot
g)\Bigr\}\cF\ii$ is given by
\begin{equation}\label{e:kernel1}
k(\tau,\xi,\lambda):=\int_{\R^n} \Bigl\{q_f(\eta-\tau,\eta)-
q_f(\eta-\tau,\xi)\Bigr\} q_g(\xi-\eta,\xi)d\eta.
\end{equation}

We are now going to estimate this kernel. The estimate of $q_g$ is
standard: for any multiindex $\alpha$, $\zeta\in\R^n$ we have (for
$D^\ga_x:=-i \partial^{\ga_1+\dots\ga_n}/\partial
x_1^{\ga_1}\dots\partial x_n^{\ga_n}$, as usual):
\begin{equation*}
|\zeta^{\alpha}q_g(\zeta,\xi,\lambda)|=\left|\int_K
e^{i\lla\zeta,x\rra}D^{\alpha}_x g(x,\xi,\lambda)\dbar x\right| \leq
\vol(K)C_{\alpha}(g)(1+|\xi|+|\lambda|^{1/m})^{-m}\/.
\end{equation*}
Since $\alpha$ is arbitrary, we see that for any $N\in \N$
\begin{equation}\label{e:qg}
|q_g(\zeta,\xi,\lambda)|\leq \wt
C_N(g)(1+|\zeta|)^{-N}(1+|\xi|+|\lambda|^{1/m})^{-m}.
\end{equation}
Next we discuss the difference
$q_f(\zeta,\eta,\lambda)-q_f(\zeta,\xi,\lambda)$. Again for a
multiindex $\alpha$ we have
\begin{multline*}
|\zeta^{\alpha}(q_f(\zeta,\eta,\lambda)-q_f(\zeta,\xi,\lambda))|
=\left|\int_K
e^{i\lla\zeta,x\rra}\{D_x^{\alpha}(f(x,\eta,\lambda)-f(x,\xi,\lambda))\}\dbar
x\right|\\
\leq \int_K\,
\sup_{t\in[0,1],\,|\beta|=1}|D_x^{\alpha}\partial_{\xi}^{\beta}f(x,\xi+t(\eta-\xi),\lambda)|\dbar
x\, |\xi-\eta|\\
\leq C
\sup_{t\in[0,1]}(1+|\xi+t(\eta-\xi)|)^{m-1}(1+|\xi+t(\eta-\xi)|+|\lambda|^{\frac{1}{m}})^{r-m}
\, |\xi-\eta|
\end{multline*}
with $C:= \vol(K)C_{N}(f)$ and $N:=\max(|\alpha|,1)$, that is,
\begin{multline}\label{e:qf}
|q_f(\zeta,\eta,\lambda)-q_f(\zeta,\xi,\lambda)|\leq
\vol(K)C_N(f)(1+|\zeta|)^{-N}|\xi-\eta|\cdot\\
\sup_{t\in[0,1]}(1+|\xi+t(\eta-\xi)|)^{m-1}(1+|\xi+t(\eta-\xi)|+|\lambda|^{\frac{1}{m}})^{r-m}.
\end{multline}
To  estimate the norm of $\Op(f)\Op(g)-\Op(f\cdot g)$ as an operator
from $H^s$ to $H^{s+m-r}$ we need to estimate the norm of the
integral operator in $L^2(\R^n)$ whose kernel is given by (see
\eqref{e:schur-est})
$$\wt k
(\tau,\xi,\lambda)=(1+|\tau|)^{s+m-r}k(\tau,\xi,\lambda)(1+|\xi|)^{-s},$$
where $k(\tau,\xi,\lambda)$ is defined in (\ref{e:kernel1}). From
(\ref{e:kernel1}), (\ref{e:qg}) and (\ref{e:qf}) we infer
\begin{multline}\label{e:i1} |\wt k(\tau,\xi,\lambda)| \leq
C_N(f)\wt C_N(g)\int
(1+|\eta-\tau|)^{-N}|\xi-\eta|\;(1+|\xi-\eta|)^{-N}\\
\cdot (1+|\xi|+|\lambda|^{\frac{1}{m}})^{-m} (1+|\tau|)^{s+m-r}\\
\cdot(1+|\xi|)^{-s} \sup_{t\in
[0,1]}(1+|\xi+t(\eta-\xi)|)^{m-1}(1+|\xi+t(\eta-\xi)|+|\lambda|^{\frac{1}{m}})^{r-m}
\dbar \eta. \end{multline}

Note that we may choose $N$ as large as we please. We now distinguish two cases.\\
 \emph{Case I}: $|\eta-\xi|\leq \frac{1}{2}|\xi|$. Then for $0\leq t\leq 1$, $\frac{1}{2}|\xi|\leq |\xi+t(\eta-\xi)|\leq\frac{3}{2}|\xi|$, and thus the integrand of the right hand side of (\ref{e:i1}) can be estimated (absorbing another constant into $ C_N(f)\wt C_N(g)$) by
 \begin{multline}\label{e:i2}
 \leq C_N(f)\wt C_N(g)(1+|\eta-\tau|)^{-N}(1+|\xi-\eta|)^{1-N}(1+|\tau|)^{s+m-r}\\
 (1+|\xi|)^{-s+m-1}(1+|\xi|+|\lambda|^{\frac{1}{m}})^{r-2m}\/.
 \end{multline}
 Using {\sc Peetre}'s Inequality (we suppress the constant), we have
 $$(1+|\tau|)^{s+m-r}(1+|\xi|)^{-s+m-1}\leq (1+|\tau-\xi|)^{|s+m-r|} (1+|\xi|)^{2m-r-1}.$$ Then (\ref{e:i2})
 \begin{multline}\label{e:i3}
 \leq  C_N(f)\wt C_N(g)(1+|\eta-\tau|)^{-N}(1+|\xi-\eta|)^{1-N}(1+|\tau-\xi|)^{|s+m-r|}\\
 (1+|\xi|)^{2m-r-1}(1+|\xi|+|\lambda|^{\frac{1}{m}})^{r-2m}.
 \end{multline}
 For $0\leq r \leq m$,
 $$(1+|\xi|)^{2m-r-1}(1+|\xi|+|\lambda|^{\frac{1}{m}})^{r-2m}\leq \begin{cases}(1+|\lambda|^{\frac{1}{m}})^{-1},
 & 2m-r-1\leq 0,\\
 (1+|\lambda|^{\frac{1}{m}})^{-m}, & 2m-r-1>0.\end{cases}$$
 Thus (\ref{e:i3})
 \begin{multline}\label{e:i4}
 \leq  C_N(f)\wt C_N(g)(1+|\eta-\tau|)^{-N}(1+|\xi-\eta|)^{1-N}(1+|\tau-\xi|)^{|s+m-r|}\\
 (1+|\lambda|)^{-\min({\frac{1}{m}},1)}.
 \end{multline}
 Again {\sc Peetre}'s Inequality (once again suppressing the constant) gives that for $N>n+1$,
 \begin{align*}
 \int_{\R^n}(1+|\eta-\tau|)^{-N}&(1+|\xi-\eta|)^{1-N}\dbar\eta\\
 &\leq\int_{\R^n} (1+|\eta|)^{-N}(1+|\xi-\eta-\tau|)^{1-N+n}\dbar\eta\\
 &\leq \int_{\R^n} (1+|\eta|)^{-n-1}(1+|\xi-\tau|)^{1-N+n}\dbar\eta.
 \end{align*}
Taking this into account and integrating the right side of
(\ref{e:i4}) over $\eta$ yields
 \begin{multline}
 \int_{|\eta-\xi|\leq\frac{1}{2}|\xi|}\cdots\dbar \eta
 \leq C_N(f)\wt C_N(g)\int (1+|\eta|)^{-n-1}\dbar\eta\\
 (1+|\xi-\tau|)^{1+n+|s+m-r|-N}(1+|\lambda|)^{-\min({\frac{1}{m}},1)}\/.
 \end{multline}
 Here we choose $N$ large enough such that $N>n+1+|s+m-r|$.\\
 \emph{Case II}: $|\eta-\xi|>\frac{1}{2}|\xi|$. Then the integrand of the right hand side of (\ref{e:i1}) is estimated by
  \begin{multline*}
 \leq C_N(f)\wt C_N(g)(1+|\eta-\tau|)^{-N}(1+|\xi-\eta|)^{m-N}(1+|\tau|)^{s+m-r}\\
 (1+|\xi|)^{-s+m-1}(1+|\lambda|^{\frac{1}{m}})^{r-2m}\/.
 \end{multline*}
 Since $\frac{1}{2}|\xi|< |\eta-\xi|$, we estimate
 $$(1+|\xi|)^{-s+m-1}\leq \begin{cases}\,\;\;\;\;\;\;1,&-s+m-1\leq0,\\
 C_{s,m}(1+|\xi-\eta|)^{-s+m-1}, &-s+m-1>0.\end{cases}$$
 Now we proceed as in Case I.

 In sum we have proved that for $N$ large enough,
 $$|\wt k(\tau,\xi,\lambda)|\leq C_N(f)\wt C_N(g)(1+|\xi-\tau|)^{-n-1}(1+|\lambda|)^{-\min({\frac{1}{m}},1)}.$$
 The lemma follows from {\sc Schur}'s test finally.
\end{proof}

\section{}\label{s:appendixB}
We shall explain a topological obstruction which excludes repeating
\textsc{Seeley}'s construction literally and which was overlooked by
various authors (see, for example, \cite{Woj:SFG}, \cite{NSSS:SBV}
and \cite{Pon:SAZ}).

Given the two rays of minimal growth $L_{\ga_j}, j=1,2$ with $\spec
a_m(x,\xi)\cap L_{\ga_j}=\emptyset$ for $x\in M, \xi\in T^*_xM,
\xi\ne 0, j=1,2$, we are guaranteed a symbol ``ingredient''
$(a_m(x,\xi)-\gl)\ii$ of order $-m$ for each $\gl\in L_{\ga_1}\cup
L_{\ga_2}$ and for $\xi\ne 0$\,. Moreover, we can find a small arc
of radius $R$ connecting the two rays such that the resulting curve
$\gG_+$ belongs to the resolvent set of $A$, as explained above.

\subsection{The problem}
It might be tempting to look for a smooth deformation and extension
$\wt a$ of $a(x,\xi)$ to $\xi=0$ in such a way that for {\em all}
$(x,\xi)\in T^*M$ one has
\[
\spec \wt a(x,\xi)\cap\gG_+=\emptyset. \] Actually, we may choose
$R>0$ such that $\spec a(x,\xi)\cap \gG_+=\emptyset$ for, say,
$|\xi|=1$. Then the problem arises whether such map
\begin{gather}\label{e:gamma-symbol}
a(x,\cdot):S^{n-1}\to \cM(N,\gG_+), \\ \cM(N,V):=\{a\in\cM(N)\mid
\spec a\cap V=\emptyset\}, V\<\C
\end{gather}
can be extended over the whole $n$-dimensional ball to a map $\wt
a:B^n\to \cM(N,\gG_+)$ in a continuous way. In the preceding, $x\in
M$ is fixed, $\dim M=n$, the fibre dimension of the Hermitian bundle
is $\dim E_x=N$, $\cM(N)$ denotes the space of $N\times N$ matrices
with complex entries,  and the matrix spaces inherit the topology of
$\C^{N^2}$. We assume that we are given a trivialization of the
cotangent bundle $T^*_xM=\R^n$ and of the fibre $E_x=\C^N$\,.

\subsection{A one-dimensional counterexample}
The most simple one-dimensional example $A:=-i\frac d{d\theta}$ on
$M=S^1\,, N=1$ refutes that naive hope. $A$ is the tangential
operator for the Cauchy-Riemann operator on  the 2-ball $\{|z|\le
1\}$. We have $a(\theta,\xi)=\xi$ with $\spec
a(\theta,\xi)=\{\xi\}$, $\spec A= \Z$,  and the imaginary line
$i\R=L_{\pi/2}\cup L_{3\pi/2}$ as spectral cut for $a(\theta,\xi),
\xi\ne 0$. Clearly, we cannot get anything useful, if we multiply
$a$ just by a cut-off function leading to
\[ \wt a(\theta,\xi)=\begin{cases} \xi \text{ for }|\xi|\ge 1,\\ 0
\text{ for }|\xi|\le \gve .\end{cases}
\]
By the Intermediate Value Theorem, for each $R\in (0,1)$ there will
always be a $\hat\xi\in (\gve,1)$ such that $\wt
a(\theta,\hat\xi)=R$. However, if we exempt only one ray, say
$L_{\pi/2}$ instead of the whole imaginary line, we {\em can} deform
the given $a(\cdot,\cdot):S^1\times (\R\setminus (-1,1)) \to
\cM(1,L_{\pi/2})$ into
\begin{equation}\label{e:seeley-deform}
\begin{matrix}
\wt a(\cdot,\cdot)&:&S^1\times \R &\too & \cM(1,L_{\pi/2}),\\
\quad &\ & (\theta,\xi) &\mapsto &
\begin{cases} \xi, & \text{ for $|\xi|\ge 1$,}\\
e^{-i(1-\xi)\frac {\pi}2}\/, & \text{ for $0\le |\xi|< 1$}.
\end{cases}
\end{matrix}
\end{equation}
Here the point is that we only require that $\wt a(\theta,\xi)$ has
no purely non-negative eigenvalues. What we did was a spectral
deformation of the original matrices (here complex numbers) into the
point $\{-i\}$. Clearly, that deformation breaks down, if we have
two rays of minimal growth forming a separating curve in $\C$: There
is no continuous path connecting $\{1\}$ and $\{-1\}$ that is not
crossing the imaginary line. The topological obstruction for $n=1$
is simply that the space $\cM(1,i\R)$ has two connected components,
$(-\infty,0), (0,\infty)$ and that $a(\theta,1), a(\theta,-1)$
belong to different components.

\subsection{The essence of the topological obstruction}
Let us muse upon the cases $n,N>1$.
Shortly, the essence of the topological difficulties
overlooked by our predecessors is the following: Without loss of
generality, let $\gG_+$ be the imaginary line $i\R$. Fix a {\em
non-trivial} smooth complex vector bundle $G$ on the sphere
$S^{n-1}$ (or on the sphere cotangent bundle $S^*M$ over the
$n$-dimensional manifold $M$ -- for simplicity, however, we shall
ignore the spatial variables). Next, we embed $G$ into a trivial
bundle $S^{n-1}\times \C^k$ for $k$ sufficiently large. Let
$\{P_\xi\}_{\xi\in S^{n-1}}$ denote the smooth family of
self-adjoint projections of $\C^k$ onto the fibers $G_\xi, \xi\in
S^{n-1}$\,.

Set $a(\xi):=2P_\xi-I:\C^k\to\C^k$ and extend it, say by homogeneity
$1$ to $\R^n$ and smooth it out in 0. Then this is an elliptic
symbol with the two imaginary half-axes being rays of minimal
growth. More precisely, we have $\spec a(\xi)=\{-1,1\}, \xi\in
S^{n-1},$ and $E_{1,\xi}=G_\xi$ and $E_{-1,\xi}=G_\xi^\perp$, where
$E_{\gl,\xi}$ denotes the linear span of the eigenvectors of
$a(\xi)$ for $\gl\in\spec a(\xi)$.

Then it is impossible to find a $k\times k$ matrix valued function
$\wt a$ on the whole $\R^n$ which coincides with $a$ outside a large
ball such that $\spec \wt a(\xi)\cap \gG_+=\emptyset$ for all
$\xi\in\R^n$: Let us assume we could. Let $\wt E_{\gL_+,\xi}=\ran
\wt P_+(\xi)$ denote the linear span of all root vectors of $\wt
a(\xi)$ for eigenvalues in the positive half plane $\gL_+\<\C$. The
family of vector subspaces of $\C^k$ is continuous and forms a
vector bundle over the unit ball $B^n$. It is trivial because the
base space is contractible, but its restriction on the $n-1$ sphere
is $G$ which is by assumption non-trivial. That is a contradiction.
So, we have a necessary condition for the construction to work.

Since \textsc{Seeley} only dealt with one ray of minimal growth,
this problem did not occur there.

Therefore, we cannot expect to be able to make the wanted extension,
respectively deformation in general. Instead of the direct (and
futile) search for a suitable modification of the principal symbol
to get a well-defined resolvent for $A$ along the spectral cut
$\gG_+$ we shall apply the symbolic calculus solely to obtain a
parametrix for $A-\gl$.

\subsection{The topology of the underlying space of hyperbolic matrices}
As a service to the reader we determine the precise homotopy type of
the matrix space $\cM(N,\gG_+)$. By deformation, we may assume that
the imaginary line is the given spectral cut for all matrices
$a(x,\xi)$ for $\xi\ne 0$. In $\C\setminus \gG_+$\,, we denote the
two complementary sectors by $\gL_\pm$\,. Then the space
$\cM(N,i\R)$ of $N\times N$ matrices with no purely imaginary
(generalized) eigenvalues decomposes into $N+1$ connected components
\begin{equation}\label{e:connected-components1}
\cM_k(N,i\R):=\{a\in\cM(N,i\R)\mid \dim \ran P^+(a)=k\},\quad
k=0,1,\dots N,
\end{equation}
where \begin{equation}\label{e:spectral-proj-N}
\begin{matrix}
P^+&:&\cM(N,i\R)=:\cE &\too& \cP(N)\\
 \quad& & a &\longmapsto & -\frac 1{2\pi i}\int_{\gG_+}(a-\gl I)\ii \,
 d\gl.
\end{matrix}
\end{equation}
Here $\cP(N)=\cup_{k=0}^N\cP_k(N)$ denotes the space of projections
(idempotent $N\times N$ matrices, fibred according to the dimension
of their ranges) and $P^+(a)$ denotes the projection onto the
generalized eigenspaces of $a$ for generalized eigenvalues in the
positive sector $\gL_+$\/.

For $k=0$ and $k=N$, the spaces $\cM_k(N,i\R)$ are homeomorphic to
the full space $\cM(N)$ of all square matrices and hence
contractible. That explains why \textsc{Seeley}'s deformation is
always possible for one ray of minimal growth, dividing $\C$ into
one sector without spectrum and one sector with all the eigenvalues,
see once again Fig. \ref{f:gamma-plus}b.

To investigate the homotopy type of $\cM_k(N,i\R)$ for $k=1,\dots
N-1$, we restrict the map \eqref{e:spectral-proj-N} to a single
component $\cM_k(N,i\R)$. We obtain a fibration of the total space
$\cM_k(N,i\R)$ as a fibre bundle over the base $\cP$ with
contractible fibre
\begin{equation*}\label{e:fiber}
(P^+)\ii\{P_0\}= \{a\in\cM(\ran P_0)\mid \spec a\< \gL_+\} \times
\{a\in\cM(\ker P_0)\mid \spec a\< \gL_-\}
\end{equation*}
for any $P_0\in \cP_k(N)$. Hence, the topological spaces, the base
$\cP_k(N)$ and the total space $\cM_k(N,i\R)$ have the same homotopy
type. By orthogonalization, it suffices to consider a projection
space made of orthogonal projections which easily can be identified
with the subspaces of $\C^N$ of dimension $k$. So we arrive at the
complex Grassmannian $\operatorname{Gr}_{\C}(N,k)$, which is known
for non-trivial homotopy, if $0<k<N$.

\bibliography{localbib}
\bibliographystyle{abbrv}

\end{document}